\documentclass[a4paper,12pt]{article}
\pdfoutput=1
\title{Semi-factorial nodal curves and N\'eron models of jacobians}

\usepackage{amsmath, amssymb, mathrsfs, amsthm, enumitem}
\usepackage[numbers]{natbib}
\usepackage[all]{xy}
\usepackage{graphicx}
\usepackage{wrapfig}
\usepackage{tikz}
\usetikzlibrary{arrows}
\usepackage{caption}
\usepackage{subcaption}
\usepackage{comment}
\usetikzlibrary{positioning}
\usepackage{hyperref}
\usepackage{bookmark}
\usepackage{float}
\usepackage[top=2.5cm, bottom=2.5cm]{geometry}

\newcommand{\Z}{\mathbb{Z}}
\newcommand{\N}{\mathbb{N}}

\newcommand{\Q}{\mathbb{Q}}

\newcommand{\X}{\mathcal{X}}

\newcommand{\Y}{\mathcal{Y}}

\renewcommand{\O}{\mathcal{O}}

\newcommand{\m}{\mathfrak{m}}
\renewcommand{\L}{\mathcal{L}}

\newcommand{\bmt}{\begin{pmatrix}}
\newcommand{\emt}{\end{pmatrix}}
\newcommand{\bsm}{\left(\begin{smallmatrix}}
\newcommand{\esm}{\end{smallmatrix}\right)}

\newcommand{\til}{\widetilde}

\DeclareMathOperator{\GL}{GL}

\DeclareMathOperator{\rk}{rk}
\DeclareMathOperator{\id}{id}

\DeclareMathOperator{\Spec}{Spec}
\DeclareMathOperator{\colim}{colim}
\DeclareMathOperator{\Proj}{Proj}

\DeclareMathOperator{\coker}{coker}

\DeclareMathOperator{\Sets}{\mathbf{Sets}}
\DeclareMathOperator{\Sch}{\mathbf{Sch}}

\DeclareMathOperator{\Frac}{Frac}

\DeclareMathOperator{\Pic}{Pic}

\DeclareMathOperator{\cll}{cl}

\DeclareMathOperator{\ZZ}{\mathcal Z}
\DeclareMathOperator{\Div}{Div}
\DeclareMathOperator{\divv}{div}
\DeclareMathOperator{\depth}{depth}

\newcommand{\ra}{\rightarrow}

\newcommand{\sub}{\subset}

\theoremstyle{definition}
\newtheorem{definition}{Definition}[section]

\newtheorem{remark}[definition]{Remark}
\newtheorem{example}[definition]{Example}
\newtheorem{construction}[definition]{Construction}

\theoremstyle{plain}
\newtheorem{proposition}[definition]{Proposition}
\newtheorem{lemma}[definition]{Lemma}
\newtheorem{theorem}[definition]{Theorem}
\newtheorem{corollary}[definition]{Corollary}

\theoremstyle{remark}

\renewcommand{\phi}{\varphi}

\author{Giulio Orecchia \\ email: \href{mailto:g.orecchia@math.leidenuniv.nl}{g.orecchia@math.leidenuniv.nl} \\
Mathematisch Instituut Leiden, \\
Niels Bohrweg 1, 2333CA Leiden, The Netherlands }

\newcounter{nootje}
\newcommand{\beq}{\begin{equation}}
\newcommand{\eeq}{\end{equation}}
\newcommand{\beqs}{\begin{equation*}}
\newcommand{\eeqs}{\end{equation*}}

\begin{document}
\maketitle
\begin{abstract}
Following P\'epin, we call a family of curves over a discrete valuation ring semi-factorial if every line bundle on the generic fibre extends to a line bundle on the total space. In the case of nodal curves with split singularities, we give a sufficient and necessary condition for semi-factoriality, in terms of combinatorics of the dual graph of the special fibre. In particular, we show that performing one blow-up with center the non-regular closed points yields a semi-factorial model of the generic fibre. 

As an application, we extend the result of Raynaud relating N\'eron models of smooth curves and Picard functors of their regular models to the case of (possibly singular) curves having a semi-factorial model.
\end{abstract}
\tableofcontents
\newpage

\section*{Introduction}

Let $S$ be the spectrum of a discrete valuation ring with fraction field $K$, and let $\X\ra S$ be a scheme over $S$. Following \citep{pepin}, we say that $\X\ra S$ is \textit{semi-factorial} if the restriction map 
$$\Pic(\X)\ra \Pic(\X_K)$$
is surjective; namely, if every line bundle on the generic fibre $\X_K$ can be extended to a line bundle on $\X$. 

We consider the case of a relative curve $\X\ra S$. In \citep{pepin}, Theorem 8.1, P\'epin proved that given a geometrically reduced curve $\X_K/K$  with ordinary singularities and a proper flat model $\X\ra S$, a semi-factorial flat model $\X'\ra S$ can be obtained after a blowing-up $\X'\ra \X$ with center in the special fibre.


The main result of this article is a necessary and sufficient condition for semi-factoriality in the case where $\X \ra S$ is a proper, flat family of nodal curves, whose special fibre has split nodes. It turns out that in this case semi-factoriality is equivalent to a certain combinatorial condition involving the dual graph of the special fibre of $\X/S$ and a labelling of its edges, which we describe now.
Let $t\in \Gamma(S,\O_S)$ be a uniformizer; every node of the special fibre is \'etale locally described by an equation of the form 
\begin{itemize}
\item[a)] $xy-t^n=0$ for some $n\geq 1$, or
\item[b)] $xy=0$ (if the node persists in the generic fibre).
\end{itemize}  
Consider the dual graph $\Gamma=(V,E)$ associated to the special fibre of $\X/S$. We label its edges by the function $l\colon E\ra \Z_{\geq 1}\cup\{\infty\}$ 
$$
l(e)=\begin{cases}
n & \mbox{ if the node corresponding to } e \mbox{ is as in case a);}\\
\infty & \mbox{ if the node corresponding to } e \mbox{ is as in case b).}
\end{cases}
$$
We say that the labelled graph $(\Gamma,l)$ is \textit{circuit-coprime} if, after contracting all edges with label $\infty$, every circuit of the graph has labels with greatest common divisor equal to $1$. In particular, if $\Gamma$ is a tree, $(\Gamma,l)$ is automatically circuit-coprime.

The following theorem is our main result:
\begin{theorem}[Theorem \ref{mainthm}] \label{Thmintro}
If the labelled graph $(\Gamma,l)$ is circuit-coprime, the curve $\X\ra S$ is semi-factorial. If moreover $\Gamma(S,\O_S)$ is strictly henselian, the converse holds as well.
\end{theorem}

The proof (of the first statement) can be subdivided in three parts:

\begin{itemize}
\item we start by constructing a chain of proper birational morphisms of nodal curves over $S$
\begin{equation}
\ldots\ra \X_n \ra \X_{n-1}\ra \ldots \ra \X_1\ra \X_0:=\X \nonumber
\end{equation}
where every arrow is the blowing-up at the reduced closed subscheme of non-regular closed points. A generalization (Prop. \ref{lemmadimerda}) of the smoothening techniques developed in \citep{BLR}, Chapter 3, allows us to show that given a line bundle $L$ on $\X_K$ there exists a positive integer $n$ such that $L$ extends to a line bundle $\mathcal L$ on $\X_n$ (Theorem \ref{extending}).
\item in the combinatorial heart of the proof, we provide a dictionary between geometry and graph theory to reduce the study of the blowing-ups $\X_n$ and line bundles on them to the study of their dual labelled graphs and integer labellings of their edges. We show that if the labelled graph $(\Gamma,l)$ of $\X/S$ is circuit-coprime,  there exists a generically trivial line bundle $\mathcal E$ on $\X_n$ such that $\mathcal L\otimes \mathcal E$ has degree $0$ on each irreducible component of the exceptional fibre of $\pi_n\colon\X_n\ra \X$. 
\item Finally, we show (Proposition \ref{formalfunction}) that the direct image $\pi_{n *}(\mathcal L\otimes E)$ is a line bundle on $\X$ (which in particular extends $L$). This relies essentially on the fact that the exceptional fibre of $\pi_n$ is a curve of genus zero. 
\end{itemize} 

As a corollary to the theorem, we refine Theorem 8.1 of \citep{pepin} in the case of nodal curves $\X/S$ with special fibre having split nodes, by explicitly describing a blowing-up with center in the special fibre that yields a semi-factorial model:

\begin{corollary}[Corollary \ref{blowuppp}]\label{corollary}
Let $\X_1\ra \X$ be the blowing-up centered at the reduced closed subscheme consisting of non-regular closed points of $\X$. Then the curve $\X_1\ra S$ is semi-factorial. 
\end{corollary}
This follows immediately, observing that $\X_1$ has circuit-coprime labelled graph. \\

Semi-factoriality is closely connected to N\'eron models of jacobians of curves. A famous construction of Raynaud (\citep{Raynaud}) shows that if $\X\ra S$ has regular total space, a N\'eron model over $S$ for the jacobian $\Pic^0_{\X_K/K}$ is given by the $S$-group scheme $\Pic^{[0]}_{\X/S}/\cll(e)$, where $\Pic^{[0]}_{\X/S}$ represents line bundles of total degree zero on $\X$, and $\cll(e)$ is the schematic closure of the unit section $e\colon K\ra \Pic^0_{\X_K/K}$. In \citep{pepin}, Theorem 9.3., it is shown that the same construction works in the case of semi-factorial curves $\X\ra S$ with smooth generic fibre.  Our second main theorem is a corollary of Theorem \ref{Thmintro}:

\begin{theorem}[Theorem \ref{mainneron}]
Let $\X\ra S$ be a nodal curve over the spectrum of a discrete valuation ring. Then $\Pic_{\X/S}/\cll(e)$ is a N\'eron lft-model over $S$ for $\Pic_{\X_K/K}$ if and only if the labelled graph $(\Gamma,l)$ is circuit-coprime.
\end{theorem}
Note that there are no smoothness assumptions on the generic fibre. The abbreviation ``lft'' stands for   ``locally of finite type'', meaning that we do not require the model to be quasi-compact (even if we chose to impose degree restrictions on $\Pic_{\X_K/K}$,  the resulting N\'eron lft-model may not be quasi-compact in general, as $\X_K/K$ may not be smooth).

\subsection*{Outline}
In section \ref{s1} we introduce the basic definitions, including that of nodal curve with split singularities. In section \ref{s2} we define an infinite chain of blow-ups of a given nodal curve $\X/S$ and then show that every line bundle on the generic fibre $\X_K/K$ extends to a line bundle on one of these blow-ups (section \ref{sectionextending}). Section \ref{s4} contains an important technical lemma on descent of line bundles along blowing-ups. 
Section \ref{sectiongraphtheory} is entirely graph-theoretic and contains the definition of circuit-coprime labelled graphs. The combinatorial results established in this section are then reinterpreted in section \ref{s6} in geometric terms in order to give a necessary and sufficient condition for semi-factoriality of nodal curves. In section \ref{s7}, starting from a nodal curve $\X/S$, we construct a N\'eron model of the Picard scheme of its generic fibre.

\subsection*{Acknowledgements}
I am extremely grateful to David Holmes for introducing me to the topic and supporting me patiently with invaluable advice and commitment.
I would also like to thank Alberto Bellardini and Bas Edixhoven for helpful comments and discussions and Raymond van Bommel for kindly devoting a whole afternoon to helping me solve some issues that arose while working out examples. Finally, I express my gratitude to the anonymous referee, for promptly producing a very accurate report pointing out several imprecisions in the first version. The overall quality of the paper benefited enormously from their comments.

\newpage

\newpage

\section{Preliminaries} \label{s1}
\subsection{Nodal curves} \label{subsectionnodalcurves}

\begin{definition} A \textit{curve} $X$ over an algebraically closed field $k$ is a proper morphism of schemes $X\ra \Spec k$, such that $X$ is connected and whose irreducible components have dimension $1$. A curve $X/k$ is called \textit{nodal} if for every non-smooth point $x\in X$ there is an isomorphism of $k$-algebras $\widehat{\O}_{\X,x}\ra k[[x,y]]/xy$.


For a general base scheme $S$, a \textit{nodal curve} $f\colon\X\ra S$ is a proper, flat morphism of finite presentation, such that for each geometric point $\overline{s}$ of $S$ the fibre $\X_{\overline s}$ is a nodal curve.
\end{definition}
We are interested in the case where the base scheme $S$ is a trait, that is, the spectrum of a discrete valuation ring. In what follows, whenever we have a trait $S$, unless otherwise specified we will denote by $K$ the fraction field of $\Gamma(S,\O_S)$ and by $k$ its residue field.

\begin{definition}
Let $X\ra \Spec k$ be a nodal curve over a field and $n\colon X'\ra X$ be the normalization morphism. A non-regular point $x\in X$ is a \textit{split ordinary double point} if the points of $n^{-1}(x)$ are $k$-rational (in particular, $x$ is $k$-rational). We say that $X\ra \Spec k$ has \textit{split singularities} if all non-regular points $x\in X$ are split ordinary double points.
\end{definition}

It is clear that the base change of a curve with split singularities still has split singularities. Also, it follows from \citep{liu}, Corollary 10.3.22 that for any nodal curve $\X\ra S$ over a trait there exists an \'etale base change of traits $S'\ra S$ such that $\X\times_SS'\ra S'$ has split singularities.

The following two lemmas are Corollary 10.3.22 b) and Lemma 10.3.11  of \citep{liu}:

\begin{lemma} \label{singularities}
Let $f\colon\X\ra S$ be a nodal curve over a trait and let $x\in \X$ be a split ordinary double point lying over the closed point $s\in S$. Write $R$ for $\Gamma(S,\O_S)$ and $\m$ for its maximal ideal. Then 
$$
\widehat{\O}_{\X, x}\cong  \frac{\widehat{R}[[x,y]]}{xy-c}
$$
for some $c\in \m R$. The ideal generated by $c$ does not depend on the choice of $c$.
\end{lemma}

We define an integer $\tau_x\in \Z_{\geq 1}\cup \{\infty\}$, given by the valuation of $c$ if $c\neq 0$ and by $\infty$ if $c=0$. We call $\tau_x$ the \textit{thickness} of $x$. The point $x$ is non-regular if and only if $\tau_x \in \mathbb{Z}_{\geq 2}\cup \{\infty\}$; moreover, $\tau_x=\infty$ if and only if $x$ is the specialization of a node of the generic fibre $\X_K$.

\begin{remark}
If the hypothesis that the special fibre has split singularities is dropped, the same result holds after replacing $R$ and $\O_{\X,x}$ by their strict henselizations.
\end{remark}

\begin{lemma} \label{twocomponents}
Let $X$ be a nodal curve over a field $k$, $x\in X$ a split ordinary double point such that at least two irreducible components of $X$ pass through $x$. Then $x$ belongs to exactly two irreducible components $Z_1,Z_2$ which are smooth at $x$ and meet transversally. 
\end{lemma}

In view of Lemma \ref{twocomponents}, if $X/k$ is a nodal curve with split singularities, the \textit{dual graph} $G$ of $X$ can be defined. The vertices of $G$ correspond to the irreducible components of $X$, while every edge $e$ between vertices $v,w$ corresponds to a an ordinary double point contained in the components corresponding to $v$ and $w$.

\subsection{Semi-factoriality}
\begin{definition}[\citep{pepin} 1.1.]\label{defsemifact}
Let $\X\ra S$ be a scheme over a trait. We say that $\X$ is \textit{semi-factorial} over $S$ if the restriction map
$$\Pic(\X)\ra \Pic(\X_K)$$
is surjective.
\end{definition}



\section{Blowing-up nodal curves}\label{s2}
Let $f\colon\X\ra S$ be a nodal curve over a trait. In this section we study the effects of blowing-up non-regular points of $\X$ lying on the special fibre of $\X\ra S$.

\subsection{Blowing-up a closed non-regular point}
\begin{lemma} \label{nodalcurves}
Let $\X\ra S$ be a nodal curve over a trait. Let $x$ be a non-regular point lying on the special fibre of $\X\ra S$. The blowing-up $\pi\colon \til{\X}\rightarrow \X$ centered at (the reduced closed subscheme given by) $x$ gives by composition a nodal curve $\til{\X}\ra S$.
\end{lemma}
\begin{proof}
The map $\pi\colon\til{\X}\ra \X$ is proper, hence so is the composition $\til{\X}\ra S$. 
Let $\overline{x}$ be a geometric point of $\X$ lying over $x$. We write $A:=\widehat{R}^{sh}$ for the completion at its maximal ideal of the strict henselization of $R$ induced by $\overline x$. Similarly we let $B:=\widehat{\O}^{\acute{e}t}_{\X,\overline x}$ be the completion of the \'etale local ring of $\X$ at $\overline x$. We have $B\cong A[[u,v]]/uv-c$ for some $c\in A$; we will assume that $c=0$, as the reader can refer to \citep{liu}, Example 8.3.53 for the case $c\neq 0$. 

The blowing-up $\mathcal Z\rightarrow \Spec B$ at the maximal ideal $\m=(t,u,v)\subset B$ fits in a cartesian diagram
\begin{center}
$
\xymatrix{
\mathcal Z \ar[r]\ar[d] & \til{\X} \ar[d]^{\pi} \\
\Spec B \ar[r] & \X
}
$
\end{center} 
with flat horizontal maps and is given by
$$\mathcal Z=\Proj \frac{B[S,U,V]}{I}$$
where $I$ is the homogenous ideal 
$$I=(uS-t U,vS-t V,uV,vU,UV).$$
The scheme $\mathcal Z$ is covered by three affine patches, given respectively by the loci where $S,U,V$ are invertible. Namely we have:
\begin{equation}
D^+(S)\cong \Spec \frac{A[U,V]}{UV}, \; D^+(U)\cong \Spec \frac{A[[u]][S]}{t-uS}, \; D^+(V)\cong \Spec \frac{A[[v]][S]}{t-vS}. \nonumber
\end{equation}
To see that $\til{\X}$ is $S$-flat, we check that the image of the uniformizer $t\in R$ is torsion-free in $\O_{\til{\X}}$, which is immediate upon inspection of the coordinate rings of $D^+(S), D^+(U), D^+(V)$. Also, for all field valued points $y\colon \Spec L\ra \Spec A$ lying over the closed point of $\Spec A$,  the completed local rings at the singular points of $\mathcal Z_y$ are of the form $L[[x,y]]/xy$, as desired.

\end{proof}

\subsection{An infinite chain of blowing-ups}
Write now $\X^{nreg}$ for the non-regular locus of $\X$. By the very definition of nodal curve, the locus $\X^{nreg}$ is a closed subset of $\X$, and in particular its intersection with the special fibre $\X_k\cap \X^{nreg}$ is a finite union of closed points. We inductively construct a chain of proper birational maps of nodal curves as follows.
\begin{construction}
Let $Y_0$ be the closed subscheme given by $\X_k\cap \X^{nreg}$ with its reduced structure.
Blowing-up $\Y_0$ in $\X$ we obtain a proper birational morphism $\pi_1 \colon\X_1\ra \X$, which restricts to an isomorphism on the dense open $\X\setminus \Y_0$ and in particular over the generic fibre. For $i\in \Z_{\geq 1}$ we let $Y_i:=(\X_i)_k\cap (\X_i)^{nreg}$ with its reduced structure, and define $\X_{i+1}\ra \X_i$ to be the blowing-up at $Y_i$. We obtain a (possibly infinite) chain of proper birational $S$-morphisms between nodal curves,
\begin{equation}(\pi_n\colon\X_n\ra \X_{n-1})_{n\in \Z_{\geq 1}}, \;\; \X_0:=\X \label{chainn} \end{equation} 
which eventually stabilizes if and only if the generic fibre $\X_K$ is regular. 
\end{construction}

\subsection{The case of split singularities} \label{blup}
From the calculations of the Lemma \ref{nodalcurves} we deduce how blowing-up alters the special fibre of a nodal curve whose special fibre has split singularities. Let $\X\ra S$ be such a curve and let $p\in \X$ be a non-regular point of the special fibre. We have $k(p)=k$. Let $\pi\colon\til{\X}\ra \X$ be the blow-up at $p$, $Y=\Spec \widehat{\O}_p$, and $\widetilde Y=Y\times_{\X}\widetilde{\X}$. Then $\pi_Y\colon \widetilde Y\ra Y$ is the blowing-up at the closed point $q$ of $Y$. Explicit calculations show that the exceptional fibre $\pi_Y^{-1}(q)$ is a chain of projective lines meeting transversally at nodes defined over $k(q)$. The equality $k(q)=k$ induces an isomorphism between the exceptional fibre of $\pi_Y\colon \til Y\ra Y$ and the exceptional fibre $\pi^{-1}(p)$ of $\pi\colon \til \X\ra \X$. 

We now distinguish all possible cases:
\begin{itemize}
\item If $\tau_p=\infty$, so that $p$ is the specialization of a node $\zeta$ of $\X_K$, $\pi^{-1}(p)$ is given by two copies of $\mathbb{P}^1_k$ meeting at a $k$-rational node $p'$ with $\tau_{p'}=\infty$;
\item if $\tau_p=2$, $\pi^{-1}(p)$ consists of one $\mathbb P^1_k$;
\item finally, if $\tau_p>2$, then $\pi^{-1}(p)$ consists again of two copies of $\mathbb P^1_k$, meeting at a $k$-rational node $p'$ with $\tau_{p'}=\tau_p-2$. 
\end{itemize}
In all cases, the intersection points between $\pi^{-1}(p)$ and the closure of its complement in $\til{\X}_k$ are regular in $\til{\X}$, that is, they have thickness $1$, and are $k$-rational. Moreover, $\til\X\ra S$ has special fibre with split singularities.

\section{Extending line bundles to blowing-ups of a nodal curve} \label{sectionextending}

Our first aim is to prove that for any line bundle $L$ on the generic fibre $\X_K$, there exists an $n\geq 0$ such that $L$ extends to a line bundle on the surface $\X_n$ of the chain of nodal curves \eqref{chainn}. In order to do this, we recall and slightly generalize the definition of N\'eron's measure for the defect of smoothness  presented in \citep{BLR}, Chapter 3.

\definition
Let $R$ be a discrete valuation ring and $\ZZ$ an $R$-scheme of finite type. Let $R\ra R'$ be a local flat morphism of discrete valuation rings. Let $a\in \ZZ(R')$ and denote by $\Omega^1_{\ZZ/R}$ the $\O_{\ZZ}$-module of $R$-differentials. The pullback $a^*\Omega^1_{\ZZ/R}$ is a finitely-generated $R'$-module, thus a direct sum of a free and a torsion sub-module. We define \textit{N\'eron's measure for the defect of smoothness} of $\ZZ$ along $a$ as
$$\delta(a):=\mbox{length of the torsion part of }a^*\Omega^1_{\ZZ/R}$$

\begin{remark}In \citep{BLR} 3.3, the measure for the defect of smoothness is defined for points with values in the strict henselization $R^{sh}$ of $R$ (which amounts to considering only local \'etale morphisms $R\ra R'$). We allow more general maps because we will need them in the proof of Theorem \ref{extending}.
\end{remark}

The following two lemmas generalize two analogous results in \citep{BLR} 3.3, concerning N\'eron's measure for the defect of smoothness to the case of points $a\in \ZZ(R')$ with $R'$ a (possibly ramified) local flat extension of $R$. In the following lemma, we denote by $\ZZ^{sm}$ the $S$-smooth locus of $\ZZ$.
\begin{lemma} \label{nerondefectzero}
Let $R$ be a discrete valuation ring and $\ZZ$ an $R$-scheme of finite type. Let $a\in \ZZ(R')$ for some local flat extension $R\ra R'$ of discrete valuation rings. Assume that the restriction to the generic fibre $a_{K'}\colon\Spec K'\ra \ZZ_{K'}$ factors through the smooth locus $\ZZ^{sm}_{K'}$ of $\ZZ_{K'}$. Then 
$$\delta(a)=0 \Leftrightarrow a\in \ZZ^{sm}(R')$$
\end{lemma}

\begin{proof}
See \citep{BLR} 3.3/1, for a proof in the case of smooth generic fibre and $R\ra R'$ a local \'etale map of discrete valuation rings. The same proof works for non-smooth generic fibre, as long as $a_K$ factors through $\ZZ^{sm}$. Now notice that $a^*\Omega_{\ZZ/R}\cong (a')^*\Omega_{\ZZ_{R'}/R'}$, where $a'\colon \Spec R'\ra \ZZ_{R'}$ is the section induced by $a$. We conclude by the fact that the smooth locus of $\ZZ/R$ is preserved under the faithfully flat base change $\Spec R'\ra \Spec R$.
\end{proof}


\begin{proposition} \label{lemmadimerda}
Let $R$ be a discrete valuation ring, $\ZZ$ an $R$-scheme of finite type, $f\colon R\ra R'$ a local flat morphism of discrete valuation rings with ramification index $r\in \Z_{\geq 1}$. Suppose $a\in \ZZ(R')$ is such that the restriction to the generic fibre $a_{K'}$ factors through the smooth locus of $\ZZ_K$. Let $\pi\colon\widetilde{\ZZ}\ra \ZZ$ be the blowing-up at the closed point $p=a\cap\ZZ_k$ with its reduced structure and denote by $\widetilde{a}\in \widetilde{\ZZ}(R')$ the unique lifting of $a$ to $\widetilde{\ZZ}$. Then 
$$\delta(\widetilde{a})\leq \max(\delta(a)-r,0).$$
\end{proposition}

\begin{proof}
For $R'=R$, Proposition \ref{lemmadimerda} is a particular case of \citep{BLR} 3.6/3. The strategy of the proof is to reduce to this case.

Denote by $t$ a uniformizer for $R$, and by $u$ a uniformizer for $R'$, with $u^r=t$ in $R'$. Since $\ZZ(R')=\ZZ_{R'}(R')$ the section $a$ can be interpreted as a section $b\in \ZZ_{R'}(R')$. Because $\Omega^1_{\ZZ_{R'}/R'}\cong \Omega^1_{\ZZ/R}\otimes_RR'$, we have $\delta(a)=\delta(b)$. The flat map $f\colon R\ra R'$ induces a cartesian diagram
\begin{center}
$\xymatrix{
\widetilde{\ZZ}_{R'} \ar[r]\ar[d]^{\pi_{R'}} & \widetilde{\ZZ} \ar[d]^{\pi}\\
\ZZ_{R'} \ar[r]^{g} & \ZZ &
}$
\end{center}
where $\pi_{R'}\colon\widetilde{\ZZ}_{R'}\ra \ZZ_{R'}$ is the blowing-up of the preimage $g^{-1}(p)$ of $p$ via $g\colon \ZZ_{R'}\ra \ZZ$.
Then the lifting $\widetilde{a}\in \ZZ(R')$ factors via the unique lifting of $b$ to $\widetilde{b}\in \widetilde{\ZZ}_{R'}(R')$. All we need to prove is that $\delta(\widetilde{b})\leq \max\{\delta(b)-r,0\}$.
We may work locally around $p$, and assume $\ZZ=\Spec A$ for some $R$-algebra $A$, and write $\ZZ_{R'}=\Spec B$ with $B=A\otimes_RR'$. We let $(t,x_1,\ldots,x_n)\subset A$ be the maximal ideal corresponding to $p$. The ideal of the closed subscheme $g^{-1}(p)\subset \ZZ_{R'}=\Spec B$ is then $I=(u^r,x_1,\ldots,x_n)\subset B$, so in particular $g^{-1}(p)$ is non-reduced for $r>1$.

We want to decompose the blowing-up $\pi_{R'}\colon \widetilde{\ZZ}_{R'}\rightarrow \ZZ_{R'}$ into a chain of $r$ blowing-ups and then apply to each of these the known case described in the beginning. We construct the chain as follows: we first blow up the ideal $I_1=(u,x_1,\ldots,x_n)\subset B$ and obtain a blowing-up map $\ZZ_1\rightarrow \ZZ_{R'}$. The scheme $\ZZ_1$ is a closed subscheme of $\mathbb P^{n}_B$, whose defining homogeneous ideal is the kernel of the map of graded $B$-algebras
$$B[u^{(1)},x_1^{(1)},\ldots,x_n^{(1)}]\rightarrow \oplus_{d\geq 0} I_1^d $$ 
given by sending $u^{(1)}$ to $u$ and $x_i^{(1)}$ to $x_i$ for all $i=1,\ldots,n$. The locus $D^+(u^{(1)})\subset \ZZ_1$ where $u^{(1)}$ does not vanish is affine, and we denote it by $\Y_1$. We blow up its closed subscheme given by the ideal $(u,x_1^{(1)}/u^{(1)},x_2^{(1)}/u^{(1)},\ldots,x_n^{(1)}/u^{(1)})$, and obtain a map
$$\ZZ_2\rightarrow \Y_1.$$
Next we consider the affine $\Y_2:=D^+(u^{(2)})\subset \ZZ_2$ and reiterating the procedure $r$ times, we end up with a chain of morphisms
$$\Y_r\rightarrow \Y_{r-1}\rightarrow \ldots \rightarrow \Y_1\rightarrow \ZZ_{R'}$$
of affine schemes.
Let's now relate this chain of maps to the blowing-up $\widetilde{\ZZ}_{R'}\rightarrow \ZZ_{R'}
$ given by the ideal $(u^r,x_1,\ldots,x_n)$. Combining the relations $$\frac{x_i^{(j-1)}}{u^{(j-1)}}u^{(j)}=ux_i^{(j)}$$ for all $j=1,\ldots,r$ (where we also set $x_i^{(0)}:=x_i$ and $u^{(0)}:=u$), we obtain in $\Y_r$ the equality
$$x_i=\frac{x_{i}^{(r)}}{u^{(r)}} u^r$$
for all $i=1,\ldots,n$. Hence the ideal sheaf $(u^r,x_1,\ldots,x_n)$ on ${\ZZ}_{R'}$ has preimage in $\Y_r$ which is free of rank $1$, generated by $u^r$. By the universal property of blowing-up we obtain a unique map $\alpha\colon\Y_r\rightarrow \widetilde{\ZZ}_{R'}$ such that the diagram 
\begin{center}
$ 
\xymatrix{
 & \widetilde{\ZZ}_{R'} \ar[d] \\
\Y_r \ar[ur]^{\alpha} \ar[r] & \ZZ_{R'} 
}
$
\end{center}
commutes. Next, we focus on the blow-up map $\widetilde{\ZZ}_{R'}\ra \ZZ_R$. The scheme $\widetilde{\ZZ}_{R'}$ is a closed subscheme of $\mathbb P^{n}_B$, whose defining homogeneous ideal is the kernel of the map of graded $B$-algebras
$$B[v,y_1,\ldots,y_n]\rightarrow \oplus_{d\geq 0} I^d $$ 
given by sending $v$ to $u^r$ and $y_i$ to $x_i$ for all $i=1,\ldots,n$. So we have relations $vx_i=u^ry_i$ for all $i=1,\ldots,n$. Then the map $\alpha^*\colon \O_{\widetilde{\X}_{R'}}\ra \O_{\Y_r}$ sends $y_i$ to $x_i^{(r)}$ and $v$ to $u^{(r)}$. We restrict our attention to the open affine $\Y\subset \widetilde{\ZZ}_{R'}$ where $v$ does not vanish. Since $v$ is mapped by $\alpha^*$ to $u^{(r)}$, which does not vanish on $\Y_r$, the map $\alpha$ factors as a map $\alpha'\colon \mathcal Y_r\ra \Y$ followed by the inclusion $\mathcal Y \subset \widetilde{\ZZ}_{R'}$. Now we produce an inverse to $\alpha'$. One checks that the ideal sheaf $(u,x_1,\ldots,x_n)$ of $\ZZ_{R'}$ becomes free in $\Y$ (generated by $u$), hence we obtain a unique map $\Y\rightarrow \Y_1$ compatible with the maps to $\ZZ_{R'}$. Then the argument can be reiterated to produce a commutative diagram 
\begin{center}
$ 
\xymatrix{
 &&&& \Y \ar[d] \ar[dllll] \ar[dlll] \ar[dl]\\
\Y_r  \ar[r] & \Y_{r-1}  \ar[r] & \ldots \ar[r] & \Y_1  \ar[r] & \ZZ_{R'} 
}
$
\end{center} 

In particular we obtain a map $\beta\colon \mathcal Y\ra \mathcal Y_r$. 
It is an easy check that the maps $\alpha'$ and $\beta$ produced between $\Y$ and $\Y_r$ are inverse one to another, hence they give an isomorphism $\Y_r\rightarrow \Y$. If we let  $b_i$ be the unique lift to $\Y_i$ of $b_0:=b\colon R'\rightarrow \ZZ_{R'}$, we obtain by \citep{BLR} 3.3/5, $\delta(b_i)\leq \min\{\delta(b_{i-1})-1,0\}$ for all $1\leq i \leq r$. Hence $\delta(\widetilde b)\leq \min\{\delta(b)-r,0\}$ as desired.

\end{proof}

We now have the tools to prove our main result on extending line bundles to blowing-ups in the chain of morphisms \eqref{chainn}.
\begin{theorem} \label{extending}
Let $L$ be a line bundle on $\X_K$. Let $(\pi_i\colon\X_i\ra \X_{i-1})_i$ be the chain \eqref{chainn} of blow-ups. Then there exists $N\geq 0$ for which $L$ extends to a line bundle $\L$ on $\X_N$. 

\end{theorem}
\begin{proof}
Let $L$ be an invertible sheaf on $\X_K$, and $D$ be a Cartier divisor with $\O_{\X_K}(D)\cong L$. We may take $D$ to be supported on the smooth locus of $\X_K$ (\citep{shaf}, Theorem 1.3.1) and see it as a Weil divisor. We may also assume that $D$ is effective, since any Weil divisor is the difference of two effective Weil divisors.  

The closed subscheme $D_{red}$ given by the support of $D$ with its reduced structure is a disjoint union of finitely many closed points of the smooth locus of $\X_K$. We write 
$$D_{red}=\bigcup_{i=1}^s P_i$$
where $P_i\in \X^{sm}_K(K_i)$ for finite extensions $K\hookrightarrow K_i$, $i=1,\ldots,s$. For each $i=1,\ldots,s$, we let $R_i$ be the localization at some prime of the integral closure of $R$ in $K_i$, so that each $R_i$ is a discrete valuation ring with fraction field $K_i$. The curve $\X/R$ being proper, each $P_i$ extends to $Q_i\in \X(R_i)$. Write $\X^{nsm}$ for the non-smooth locus of $\X/R$ and $\X^{nreg}$ for the non-regular locus of $\X$. Notice that $\delta(Q_i)>0$ if and only if $Q_i\cap \X_k\in \X^{nsm}$, by Lemma \ref{nerondefectzero}. Assume that the point $Q_i\cap \X_k$ lies in $\X^{nreg}\sub \X^{nsm}$. In this case, it is one of the closed points that are the center of the blowing-up $\X_1\ra \X$. By Proposition \ref{lemmadimerda}, the unique lifting $Q'_i$ of $Q_i$ to $\X_1$ satisfies $\delta(Q'_i)\leq \max(0,\delta(Q_i)-r_i)$, where $r_i\geq 1$ is the ramification index of $R\ra R_i$. Letting $Q_i^{(n)}$ be the lifting of $Q_i$ to $\X_n$, we see that either $\delta(Q_i^{(n)})\leq \delta(Q_i^{(n-1)})-r_i$ or $Q_i^{(n-1)}\subset (\X^{sm}\subset) \X^{reg}$. Hence, for $N$ big enough, each of the points $P_i\in \X_K$ extends to $Q_i^{(N)}\in\X^{reg}_N(R_i)$. Therefore the Weil divisor $D$ extends to a Weil divisor $\widetilde{D}$ on $\X_N$ that is supported on the union of the $Q_i^{(N)}$, hence on the regular locus of $\X_N$. This implies that $\widetilde{D}$ is a Cartier divisor, and the line bundle $\O_{\X_n}(\widetilde{D})$ restricts to $\O_{\X_K}(D)\cong L$ on $\X_K$. This completes the proof.

\end{proof}

\section{Descent of line bundles along blowing-ups} \label{s4}
\begin{lemma} \label{pushf}
Let $S$ be a trait and $\pi\colon \mathcal Y \ra \mathcal X$ a proper morphism of flat $S$-schemes, which restricts to an isomorphism over the generic point of $S$. Assume that the special fibre $\X_k$ is reduced. Then $\pi_*\O_{\mathcal Y}\cong \O_{\X}$.
\end{lemma}

\begin{proof}

Consider an affine open $W\sub \X$. The morphism $\O_{\X}(W)\ra \pi_*\O_{\mathcal Y}(W)$ is integral (\citep{liu}, Prop.3.3.18). Denoting by $t$ a uniformizer of $\Gamma(S,\O_S)$, we have a commutative diagram 

\begin{center}
$ \xymatrix{
\O_{\X}(W)\ar[r]\ar[d] & \pi_*\O_{\Y}(W)\ar[d] \\
\O_{\X}(W)[t^{-1}] \ar[r]^{\cong} & (\pi_*\O_{\Y}(W))[t^{-1}]
}
$
\end{center}
The two vertical arrows are injective because $\mathcal X$ and $\mathcal Y$ are $S$-flat; the lower arrow is an isomorphism because $\pi$ is generically an isomorphism and $(\pi_*\O_{\Y}(W))[t^{-1}]=\pi_*(\O_{\Y}(W)[t^{-1}])$. It follows that the upper arrow is injective. We claim that $\O_{\X}(W)$ is integrally closed in $\O_{\X}(W)[t^{-1}]$, so that the upper arrow is an isomorphism, which proves the lemma. Take then $g\in \O_{\X}(W)[t^{-1}]$ satisfying a monic polynomial equation $g^m+a_1g^{m-1}+\ldots+a_m=0$ with coefficients in $\O_{\X}(W)$ and write $g=f/t^n$ with $f\in \O_{\X}(W)$ and $n\geq 0$ minimal. We want to show that $n$ is zero. We have 
$$\frac{f^{m}}{t^{nm}}+a_1\frac{f^{m-1}}{t^{n(m-1)}}+\ldots+a_m=0.$$ Suppose by contradiction $n\geq 1$. Upon multiplying by $t^{nm}$ the above relation, we find that $f^m\in t\O_{\X}(W)$. Because the special fibre of $\X$ is reduced, the ring $\O_{\X}(W)/t\O_{X}(W)$ is reduced, hence $f \in t\O_{\X}(W)$. This violates the hypothesis of minimality of $n$ and we have a contradiction. Hence $n=0$ and $g\in \O_{\X}(W)$, proving the claim.

\end{proof}

\begin{proposition} \label{formalfunction}
Hypotheses as in Lemma \ref{pushf}. Let $\mathcal L$ be a line bundle on $\Y$ such that its restriction to every irreducible component of the exceptional fibre of $\pi$ has degree zero. Then $\pi_*\mathcal L$ is a line bundle on $\X$. 
\end{proposition}

\begin{proof}
We first consider the case where $S$ is the spectrum of a strictly henselian discrete valuation ring. In this case, the special fibre of $\X\ra S$ has split singularities, hence, as seen in subsection \ref{blup}, the exceptional fibre $E$ of $\mathcal Y\ra \mathcal X$ consists either of a projective line, or of two projective lines meeting at a $k$-rational node.

The sheaf $\pi_*\mathcal L$ is a coherent $\O_{\X}$-module. Since the curve $\X$ is reduced, to show that $\pi_*\mathcal L$ is a line bundle it is enough to check that $\dim_{k(x)}\pi_*\mathcal L\otimes_{\O_{\X}} k(x)=1$ for all $x\in \X$. This clearly holds for $x\in \X$ different from $p$. We remain with the case $x=p$. Denote by $\O_p$ the local ring of $\X$ at $p$. Let 
$$\mathcal Z:=\Y\times_{\X}\Spec \mathcal O_p$$
so that $\mathcal Z$ is the blow-up of $\Spec \O_p$ at its closed point. We write $\mathcal I$ for the ideal sheaf $\m_p\O_{\mathcal Z}\subset \O_{\mathcal Z}$. For every $n\geq 1$ define
$$\mathcal Z_n:=\Y\times_{\X}\Spec \O_{p}/\m_p^n$$ 
so we have $\O_{\mathcal Z_n}=\O_{\mathcal Z}/\mathcal I^n$.
In particular, $\mathcal Z_1$ is the exceptional fibre of the blowing-up $\mathcal Z\ra \Spec \O_p$, which coincides with the exceptional fibre $E$ of $\pi\colon\mathcal Y\ra \mathcal X$. 

The formal function theorem tells us that there is a natural isomorphism
$$\Phi\colon\lim_n (\pi_*\mathcal L) \otimes_{\mathcal O_{\X}}\mathcal O_p/\m_p^n \ra \lim_n H^0(\mathcal Z_n,\mathcal L_{|\mathcal{Z}_n}).$$

We claim that $\mathcal L_{|\mathcal Z_n}$ is trivial for all $n\geq 1$. We start with the case $n=1$: the dual graph of the curve $\mathcal Z_1$ is a tree, hence $\Pic(\mathcal Z_1)$ is the product of the Picard groups of the components of $\mathcal Z_1$ (this can be checked via the Mayer-Vietoris sequence for $\mathcal O^{\times}$, for example). In other words, a line bundle on $\mathcal Z_1$ is determined by its restrictions to the components of $\mathcal Z_1$. As $\Pic(\mathbb P^1_k)=\mathbb Z$ via the degree map, we have $\mathcal L_{|\mathcal Z_1}=\mathcal O_{\mathcal Z_1}$.
Now let $n\geq 1$ and assume that $\mathcal L_{|\mathcal Z_n}$ is trivial. There is an exact sequence of sheaves of groups on $\mathcal Z$
$$0\ra \mathcal I^n/\mathcal I^{n+1} \ra (\O_{\mathcal Z}/\mathcal I^{n+1})^{\times} \ra (\O_{\mathcal Z}/\mathcal I^n)^{\times} \ra 1$$
with the first map sending $\alpha$ to $1+\alpha$. The ideal sheaf $\mathcal I$ is canonically isomorphic to the invertible sheaf $\mathcal O_{\mathcal Z}(1)$. Hence $\mathcal I^n/\mathcal I^{n+1}=\mathcal O_{\mathcal Z_1}(n)$. Taking the long exact sequence of cohomology we obtain 
$$H^1(\mathcal Z_1,\O_{\mathcal Z_1}(n))\ra H^1(\mathcal Z_{n+1},\O_{\mathcal Z_{n+1}}^{\times})\ra H^1(\mathcal Z_{n},\O_{\mathcal Z_{n}}^{\times})\ra 0.$$
We find that the term $H^1(\mathcal Z_1,\O_{\mathcal Z_1}(n))$ vanishes using Mayer-Vietoris exact sequence and the fact that $H^1(\mathbb P^1_k,\mathcal O_{\mathbb P^1_k}(n))=0$. It follows that the restriction map 
$\Pic(\mathcal Z_{n+1})\ra \Pic(\mathcal Z_n)$ is an isomorphism. Since the sheaf $\mathcal L_{|\mathcal Z_{n+1}}$ restricts to the trivial sheaf on $\mathcal Z_n$, it is itself trivial, establishing the claim.

We obtain
$$\lim_n H^0(\mathcal Z_n,\mathcal L_{|\mathcal Z_n})\cong \lim_n H^0(\mathcal Z_n,\O_{\mathcal Z_n})\cong\lim_n(\pi_*\mathcal O_{\mathcal Y})\otimes_{\O_{\X}}\O_p/\m_p^n\cong \widehat{\O}_p$$
the second isomorphism coming again from the formal function theorem applied to $\mathcal O_{\mathcal Y}$ and the third coming from Lemma \ref{pushf}.
Finally, we obtain by composition with $\Phi$ an isomorphism
$$\lim_n(\pi_*\mathcal L)\otimes_{\O} \O_p/\m_p^n\ra \widehat{\O}_p$$ which induces an isomorphism $\pi_*\mathcal L\otimes_{\O}\O_p/\m_p \ra \O_p/\m_p=k(p)$, as desired.

Now we drop the assumption of strict henselianity on the base, so let $S$ be the spectrum of a discrete valuation ring. Let $S'$ be the \'etale local ring of $S$ with respect to some separable closure of the residue field of $S$. The cartesian diagram

\begin{center}
$
\xymatrix{
\mathcal Y_{S'} \ar[r]^{f}\ar[d]^{\pi '} & \mathcal Y \ar[d]^{\pi} \\
\mathcal X_{S'} \ar[r]^g & \mathcal X 
}
$
\end{center}
has faithfully flat horizontal arrows, and $\mathcal Y_{S'}\rightarrow \mathcal X_{S'}$ is the blowing-up at $g^{-1}(p)$. Let $\mathcal L$ be a line bundle on $\mathcal Y$ as in the hypotheses. The restrictions of $f^*\mathcal L$ to the irreducible components of the exceptional fibre of $\pi '$ have degree zero, hence $\pi'_*f^*\mathcal L$ is a line bundle. Moreover the canonical map
$$g^*\pi_*\mathcal L \rightarrow \pi'_*f^*\mathcal L$$
is an isomorphism, because $g$ is flat. Hence $g^*\pi_*\mathcal L$ is a line bundle, and so is $\pi_*\mathcal L$ by faithful flatness of $g$.
\end{proof}

\section{Graph theory} \label{sectiongraphtheory}
In this section we develop some graph-theoretic results that, together with the results of sections \ref{sectionextending} and \ref{s4}, will be needed to prove Theorem \ref{mainthm}.  

\subsection{Labelled graphs}
Let $G=(V,E)$ be a connected, finite graph. For the whole of this section, we will just write ``graph'' to mean finite, connected graph. A \textit{circuit} in $G$ is a closed walk in $G$ all of whose edges and vertices are distinct except for the first and last vertex. A \textit{path} is an open walk all of whose edges and vertices are distinct. 

A \textit{tree} of $G$ is a connected subgraph $T\subset G$ containing no circuit. A \textit{spanning tree} of $G$ is a tree of $G$ containing all of the vertices of $G$, that is, a maximal tree of $G$. Given a spanning tree $T\sub G$, we call \textit{links} the edges not belonging to $T$.

Let $n=|E|$, $m=|V|$. Given a spanning tree $T$, the number of links of $T$ is easily seen to be $n-m+1$. The number $$r:=n-m+1$$ is called \textit{nullity} of $G$ and is equal to the first Betti number $\rk H^1(G,\Z)$. 

Fix a spanning tree $T\subset G$. For each link $c_1,\ldots,c_r$ of $T$, the subgraph $T\cup c_i$ contains exactly one circuit $C_i\subset G$. We call $C_1,\ldots,C_r$ \textit{fundamental circuits} of $G$ (with respect to $T$).

Let $(G,l)=(V,E,l)$ be the datum of a graph and of a labelling of the edges $l\colon E\ra \Z_{\geq 1}$ by positive integers. We say that $(G,l)$ is a \textit{$\N$-labelled graph}.

\subsection{Circuit matrices} \label{sectioncircuitmatrices}
Given a graph $G$, let $e_1,e_2,\ldots,e_n$ be its edges and $\gamma_1,\ldots,\gamma_s$ its circuits. Fix an arbitrary orientation of the edges of $G$, and an orientation of each circuit (that is, one of the two travelling directions on the closed walk). 

\begin{definition}
The \textit{circuit matrix} of $G$ is the $s\times n$ matrix $M_G$ whose entries $a_{ij}$ are defined as follows:
$$
a_{ij}=\begin{cases}
0 & \mbox{ if the edge } e_j \mbox{ is not in } \gamma_i;\\ 
1 & \mbox{ if the edge } e_j \mbox{ is in } \gamma_i \mbox{ and its orientation agrees }\\
& \mbox{ with the orientation of } \gamma_i;\\
-1 & \mbox{ if the edge } e_j \mbox{ is in } \gamma_i \mbox{ and its orientation does not agree } \\
& \mbox{ with the orientation of } \gamma_i.\\
\end{cases}$$
\end{definition}
Hence every row of $M_G$ corresponds to a circuit of $G$ and each column to an edge.

Now fix a spanning tree of $G$. Let $c_1,\ldots,c_r$ be the corresponding links, where $r$ is the nullity of $G$, and $C_1,\ldots,C_r$ the associated fundamental circuits. Consider the $r\times n$ submatrix $N_G$ of $M_G$ given by singling out the rows 
corresponding to fundamental circuits.
One can reorder edges and circuits so that the $j$-th column corresponds to the link $c_j$ for $1\leq j \leq r$ and that the $i$-th row 
corresponds to the circuit $C_i$. If we also choose the orientation of every fundamental circuit $C_i$ so that it agrees with the orientation of the link $c_i$, the matrix $N_G$ has the form
$$N_G=[\mathbb{I}_r|N']$$
where $\mathbb{I}_r$ is the identity $r\times r$-matrix and $N'$ is an integer matrix.

\begin{definition}
The matrix $N_G$ constructed above is called the \textit{fundamental circuit matrix} of $G$ (with respect to the spanning tree $T$).
\end{definition}

It is clear that $N_G$ has rank $r$. 
\begin{theorem}[\citep{TS}, Theorem 6.7.] \label{rankMGNG}
The rank of $M_G$ is equal to the rank of $N_G$.
\end{theorem}

Let now $(G,l)$ be an $\N$-labelled graph. We generalize the definitions above to this case.
\begin{definition}
The \textit{labelled circuit matrix} of $(G,l)$ is the $s\times n$ matrix $M_{(G,l)}$ whose entries $b_{ij}$ are defined as follows:
$$
b_{ij}=\begin{cases}
0 & \mbox{ if the edge } e_j \mbox{ is not in } \gamma_i;\\ 
l(e_j) & \mbox{ if the edge } e_j \mbox{ is in } \gamma_i \mbox{ and its orientation agrees }\\
& \mbox{with the orientation of } \gamma_i;\\
-l(e_j) & \mbox{ if the edge } e_j \mbox{ is in } \gamma_i \mbox{ and its orientation does not agree }\\
& \mbox{ with the orientation of } \gamma_i.\\
\end{cases}$$

The \textit{labelled fundamental circuit (lfc) matrix} of $(G,l)$ is the $r\times n$ matrix $N_{(G,l)}$ constructed from $M_{(G,l)}$ by taking only the rows corresponding to fundamental circuits with respect to a given spanning tree $T$. 
\end{definition}
We immediately see that 
$$M_{(G,l)}=M_G\cdot L \mbox{ and } N_{(G,l)}=N_G\cdot L $$
where $L$ is the diagonal square matrix of order $n$ whose $(i,i)$-th entry is $l(e_i)$. 

\begin{example}
Consider the $\mathbb N$-labelled graph $(G,l)$ with oriented edges in Figure \ref{fig:graph}.

\begin{figure}[H]  
\centering

\begin{tikzpicture}[->,>=stealth',shorten >=1pt,auto,node distance=2cm,main node/.style={circle,draw,fill=black,inner sep=1.5,minimum size=0.4pt},every loop/.style={min distance=7mm}]

  \node[main node] (1) {};
  \node[main node] (2) [below left of=1] {};
  \node[main node] (3) [below right of=2] {};
  \node[main node] (4) [below right of=1] {};

  \path[every node/.style={font=\sffamily\small}]
    (1) edge node [left] {2} (4)
             
    (2) edge node [right] {3} (1)
        edge node {6} (4)       
    (3) edge node [right] {10} (2)
    (4) edge node [left] {15} (3)
       ;
\end{tikzpicture} 
\caption{An oriented $\N$-labelled graph $(G,l)$ \label{fig:graph}}

\end{figure}
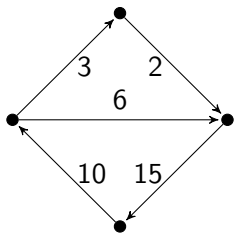

We assign to each of its three circuits the clockwise travelling direction. We obtain a circuit matrix of $G$ and a labelled circuit matrix of $(G,l)$:
\vspace{0.5cm}

\begin{minipage}[b]{0.5\textwidth}
$M_G=\begin{bmatrix}
1 & 1 & -1 & 0 & 0\\
0 & 0 & 1 & 1 & 1\\
1 & 1 & 0 & 1 & 1
\end{bmatrix} $
\end{minipage}
\begin{minipage}[b]{0.5\textwidth}
$M_{(G,l)}=\begin{bmatrix}
3 & 2 & -6 & 0 & 0\\
0 & 0 & 6 & 15 & 10\\
3 & 2 & 0 & 15 & 10
\end{bmatrix} $
\end{minipage}

\vspace{0.5cm}
Choose the spanning tree with edges labelled by $3, 6$ and $10$. The fundamental circuit matrix of $G$ and lfc-matrix of $(G,l)$ are obtained from $M_G$ and $M_{(G,l)}$ by removing the third row:
\vspace{0.5cm}

\begin{minipage}[b]{0.5\textwidth}
$N_G=\begin{bmatrix}
1 & 1 & -1 & 0 & 0\\
0 & 0 & 1 & 1 & 1
\end{bmatrix} $
\end{minipage}
\begin{minipage}[b]{0.5\textwidth}
$N_{(G,l)}=\begin{bmatrix}
3 & 2 & -6 & 0 & 0\\
0 & 0 & 6 & 15 & 10
\end{bmatrix} $
\end{minipage}

\end{example}

\vspace{1cm}

Let $M$ be an integer-valued matrix with $a$ rows and $b$ columns. There exist matrices $A\in \GL(a,\Z)$ and $B\in \GL(b,\Z)$ such that $$
AMB=\begin{bmatrix}
d_1    & 0   & 0      &     & \dots & & 0 \\
0      & d_2 & 0      &     & \dots & & 0 \\
0      & 0   & \ddots &     &       & & 0 \\
\vdots &     &        & d_k &       & & \vdots \\
       &     &        &     & 0     & &   \\
       &     &        &     &       &\ddots & \\
0      &     &        &\dots&       & & 0 \\ 
\end{bmatrix}
$$
where the diagonal entries satisfy $d_i|d_{i+1}$ for $i=1,\ldots,k-1$. This is the so-called \textit{Smith normal form} of $M$ and it is unique up to multiplication of the diagonal entries by units of $\Z$. For $1\leq i \leq k$, the integer $d_i$ is the quotient $D_i/D_{i-1}$, where $D_i$ equals the greatest common divisor of all minors of order $i$ of $M$.

Going back to the matrices $M_{(G,l)}$ and its submatrix $N_{(G,l)}$, it follows from Theorem \ref{rankMGNG} that their Smith normal forms both have rank equal to the nullity $r$ of the graph $G$. Besides, as any row of $M_{(G,l)}$ is a $\mathbb Z$-linear combination of rows of $N_{(G,l)}$, we see that the numbers $D_i$ defined above are the same for the two matrices. It follows that $M_{(G,l)}$ and $N_{(G,l)}$ have the same non-zero numbers $d_i$ appearing on the diagonal. Moreover, the numbers $d_1,\ldots,d_r$ are defined up to multiplication by $-1$, hence do not depend on the choices of orientation of edges or circuits, but only on the $\N$-labelled graph $(G,l)$.

\subsection{Cartier labellings and blow-up graphs} \label{subsCLBG}
Let $(G,l)$ be an $\N$-labelled graph. Let $\Z^{V}$ be the free abelian group generated by the set of vertices $V$. Any element $\phi$ of $\Z^V$ can be interpreted as a vertex labelling $\phi\colon V\ra \Z$ of the graph $G$.
\begin{definition} \label{cartierlabel}
An element $\phi\in \Z^V$ is a \textit{Cartier vertex labelling} if for every edge $e\in E$ with endpoints $v,w\in V$, $l(e)$ divides $\phi(v)-\phi(w)$. 
\end{definition} 
We denote by $\mathcal C\subset \Z^V$ the subgroup of Cartier vertex labellings. 
\begin{definition} \label{defdelta}
We call \textit{multidegree operator} the group homomorphism $\delta\colon \mathcal C \ra  \Z^V$ which sends $\varphi\in \mathcal C$ to
$$
v\mapsto \sum_{\substack{\mbox{\scriptsize edges }  e \\ \mbox{\scriptsize  incident to }v }} \frac{\phi(w)-\phi(v)}{l(e)} 
$$
where $w$ denotes the other endpoint of $e$ (which is $v$ itself if $e$ is a loop).
\end{definition}

\begin{lemma}\label{kernelZ} 
The kernel of $\delta$ consists of the constant vertex labellings, hence there is an exact sequence
$$0\ra \Z\ra \mathcal C \xrightarrow{\delta} \Z^V.$$
\end{lemma}
\begin{proof}Any constant vertex labelling is in the kernel of $\delta$. Conversely, let $\phi\in \ker\delta$ and let $v\in V$ be a vertex where $\phi$ attains its maximum. Then for all the vertices $w$ adjacent to $v$ one has $\phi(w)=\phi(v)$. Since the graph is finite and connected, one can repeat the argument and find that $\phi$ is a constant labelling. 
\end{proof}
\begin{remark}
When the edge-labelling $l\colon E\ra \Z_{\geq 1}$ is constant with value $1$, the multidegree operator $\delta$ coincides with the Laplacian operator of the graph $G$.
\end{remark}

\begin{definition}
Given an $\N$-labelled graph $(G,l)=(V,E,l)$ we define the \textit{total blow-up graph} $(\til G, \til l)=(\til V,\til E,\til l)$ to be the $\N$-labelled graph constructed as follows starting from $(G,l)$: every edge $e\in E$ is replaced by a path consisting of $l(e)$ edges, and $\til l\colon \til E\ra \Z$ is set to be the constant labelling with value $1$. 

\begin{example} Figure \ref{graphandblowup} shows an $\mathbb N$-labelled graph (\subref{graph1}) and its total blow-up graph (\subref{graph2}).
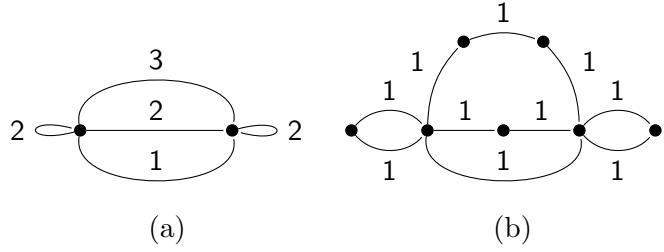
\begin{figure}[H]
 \centering
 \begin{subfigure}[b]{0.3\textwidth}
\begin{tikzpicture}[>=stealth',shorten >=1pt,auto,node distance=2cm,main node/.style={circle,draw,fill=black,inner sep=1.5,minimum size=0.4pt},every loop/.style={min distance=7mm}]

  \node[main node] (1) {};
  \node[main node] (2) [right of=1] {};

  \path[every node/.style={font=\sffamily\small}]    
      (1) edge node {2} (2)
      edge [loop left] node {2} (1)  
      edge [bend right=100]  node {1} (2)
      edge [bend left=100] node  {3} (2)
      (2) edge [loop right] node {2} (2)
       ;
\end{tikzpicture}

\caption{  \label{graph1} }
\end{subfigure}
\begin{subfigure}[b]{0.3\textwidth}
\begin{tikzpicture}[>=stealth',shorten >=1pt,auto,node distance=1cm,main node/.style={circle,draw,fill=black,inner sep=1.5,minimum size=0.4pt},every loop/.style={min distance=7mm}]

  \node[main node] (1) {};
  \node[main node] (2) [right of=1] {};
  \node[main node] (3) [right of=2] {};
  \node[main node] (4) [right of=3] {};
  \node[main node] (5) [above right=1.5cm of 2, xshift=-0.7cm] {};
  \node[main node] (6) [above left=1.5cm of 4, xshift=0.7cm] {};
  \node[main node] (7) [right of=4] {};

  \path[every node/.style={font=\sffamily\small}]
  (1) edge [bend left=50] node {1} (2)
  (1) edge [bend right=50] node [below] {1} (2)
  (2) edge [bend right=100] node {1} (4)
  (2) edge node {1} (3)
  (2) edge [bend left=20] node {1} (5)
  (3) edge node {1} (4)
  (5) edge [bend left=20] node {1} (6)
  (6) edge [bend left=20] node {1} (4)
  (4) edge [bend left=50] node {1} (7)
  (4) edge [bend right=50] node [below] {1} (7)
  ;
\end{tikzpicture}
\caption{ \label{graph2}}
\end{subfigure}
\caption{An $\N$-labelled graph $G$ (\subref{graph1}) and its total blow-up graph $\til G$ (\subref{graph2}). \label{graphandblowup}}
\end{figure}

\end{example}

We call \textit{old vertices} the vertices in the image of the  inclusion map $V\hookrightarrow \til V$. We call \textit{new vertices} the remaining vertices.
\end{definition}

Notice that every new vertex is incident to exactly two edges, and belongs to a unique path (corresponding to some edge $e\in E$) connecting two old vertices of $\til V$. Just as before we consider the group of Cartier vertex labellings $\til {\mathcal C}$ of $(\til G, \til l)$, and the multidegree operator $\til \delta\colon \til {\mathcal C}\ra \Z^{\til V}$. 

We obtain a morphism of exact sequences

\begin{gather}
\begin{aligned}
 \xymatrix{
0\ar[r] & \Z\ar[r]\ar[d]^{\id} & \mathcal C \ar[r]^{\delta}\ar[d]^{\iota} & \Z^V \ar[d]^{\epsilon} \\
0 \ar[r] & \Z\ar[r] & \til{\mathcal C} \ar[r]^{\til \delta} & \Z^{\til V} \\
} 
\end{aligned}
\label{diagramCartier}
\end{gather}

The map $\epsilon\colon \Z^V\ra \Z^{\til V}$ is given by extending vertex-labellings by zero on the set of new vertices. The map $\iota\colon \mathcal C\ra \til{\mathcal C}$ sends a Cartier vertex labelling $\phi$ on $G$ to the Cartier vertex labelling $\iota(\phi)$ on $\til G$ whose value at old vertices in inherited by $\varphi$, and extended by linear interpolation to the new vertices. More precisely: if $e$ is an edge of $G$ with endpoints $v,w$ which is replaced in $\til G$ by a path consisting of vertices $v=v_0,v_1,\ldots,v_{l(e)}=w$, we set for each $k=0,\ldots,l(e)$ 
$$\iota(\phi)(v_k)= \frac{(l(e)-k)\phi(v)+k\phi(w)}{l(e)}.$$ The Cartier condition on $\varphi$ implies that this labelling takes integer values.

Let $H=\coker \delta$, $\til H=\coker \til\delta$. The commutative diagram above yields a group homomorphism $\overline{\epsilon}\colon H\ra \til H$. 
\begin{lemma} \label{injection}
The group homomorphism $\overline{\epsilon}\colon H\ra \til{H}$ is injective.
\end{lemma}
\begin{proof}
Let $\alpha\in\Z^{V}$ be a vertex labelling and let $\epsilon(\alpha)\in \Z^{\til V}$ be its extension by zero. Assume that there exists a Cartier vertex labelling $\til \phi\in \til{\mathcal C}$ such that $\epsilon(\alpha)=\til\delta(\til \phi)$. Then $\til\delta(\til \phi)$ takes value zero on all new vertices of $\til G$. Hence, if $v$ is a new vertex of $\til G$ adjacent to two verteces $v'$ and $v''$, we have $\til\phi(v')-\til\phi(v)=\til\phi(v)-\til\phi(v'')$. We immediately see that $\til\phi$ is an interpolation of a Cartier vertex labelling $\phi \in \mathcal C$, i.e. $\til \phi$ is in the image of $\iota$.  Since $\epsilon\colon \Z^V\ra \Z^{\til V}$ is injective, $\alpha=\delta(\phi)$.
\end{proof}

Our aim now is to give necessary and sufficient conditions on the $\N$-labelled graph $(G,l)$ for the map $\overline{\epsilon}\colon H\ra \til H$ to be surjective (hence an isomorphism).

\subsection{Circuit-coprime graphs}
\begin{definition}\label{defcc}
Let $(G,l)=(V,E,l)$ be an $\N$-labelled graph. We say that $(G,l)$ is \textit{circuit-coprime} if for every circuit $C\subset G$, $\gcd\{l(e)| e \mbox{ is an edge of } C\}=1.$ 
\end{definition}
\begin{example}
In Figure \ref{excircuitcoprime} the $\mathbb N$-labelled graph (\subref{excc1}) is circuit-coprime, whereas the $\mathbb N$ -labelled graph (\subref{excc2}) is not, as it contains a loop labelled by $3$ in addition to a circuit labelled by $6,10$ and $10$.

\begin{figure}[h]
 \centering
\begin{subfigure}[b]{0.4\textwidth}
\begin{tikzpicture}[>=stealth',shorten >=1pt,auto,node distance=2cm,main node/.style={circle,draw,fill=black,inner sep=1.5,minimum size=0.4pt},every loop/.style={min distance=7mm}]

  \node[main node] (1) {};
  \node[main node] (2) [below left of=1] {};
  \node[main node] (3) [below right of=2] {};
  \node[main node] (4) [below right of=1] {};

  \path[every node/.style={font=\sffamily\small}]
    (1) edge node [left] {2} (4)
             
    (2) edge node [right] {3} (1)
        edge node {6} (4)
        edge [loop left] node {1} (2)
    (3) edge node [right] {10} (2)
    (4) edge node [left] {15} (3)
       ;
\end{tikzpicture} 
\caption{\label{excc1}}
\end{subfigure} \hspace{1.5cm}
\begin{subfigure}[b]{0.4\textwidth}
\begin{tikzpicture}[>=stealth',shorten >=1pt,auto,node distance=2cm,main node/.style={circle,draw,fill=black,inner sep=1.5,minimum size=0.4pt},every loop/.style={min distance=7mm}]

  \node[main node] (1) {};
  \node[main node] (2) [below left of=1] {};
  \node[main node] (3) [below right of=2] {};
  \node[main node] (4) [below right of=1] {};

  \path[every node/.style={font=\sffamily\small}]
    (1) edge node [left] {2} (4)
             
    (2) edge node [right] {3} (1)
        edge node {6} (4)
        edge [loop left] node {3} (2)
    (3) edge node [right] {10} (2)
    (4) edge node [left] {10} (3)
       ;
\end{tikzpicture} 
\caption{\label{excc2}}
\end{subfigure}
\caption{A circuit-coprime $\mathbb N$-labelled graph (\subref{excc1}) and an $\mathbb N$ -labelled graph that is not circuit-coprime (\subref{excc2}).\label{excircuitcoprime}}
\end{figure}
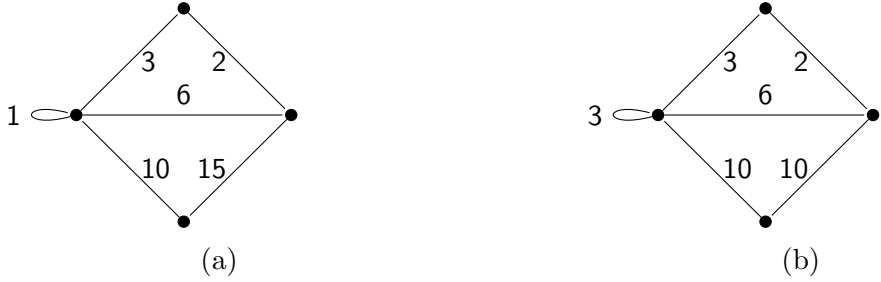
\end{example}

\begin{lemma} \label{SNFcoprime}
Let $(G,l)=(V,E,l)$ be an $\N$-labelled graph. Denote by $r$ its nullity. The Smith normal form of the matrix $M_{(G,l)}$ has diagonal entries $d_1=d_2=\ldots=d_r=1$ if and only if $(G,l)$ is circuit-coprime.
\end{lemma}
\begin{proof}
 

Assume first that $(G,l)$ is not circuit-coprime. Let $C$ be a circuit whose labels have greatest common divisor $D\neq 1$. Pick an edge $e$ of $C$. The subgraph $C\setminus e $ is a tree; let $T$ be a spanning tree of $G$ containing it. Then $e$ is a link for $T$, and $C$ is its associated fundamental circuit. The lfc-matrix $N_{(G,l)}$ has a row corresponding to the circuit $C$, hence all entries of this row are divisible by $D$. Then the linear map $f\colon \Z^n\ra \Z^r$ defined by $N_{(G,l)}$ is not surjective; hence the linear map associated to the Smith normal form of $N_{(G,l)}$ is not surjective either. Therefore, some (necessarily non-zero) diagonal entry of the Smith normal form of $N_{(G,l)}$ is different from $\pm 1$. As previously remarked, the Smith normal forms of $M_{(G,l)}$ and $N_{(G,l)}$ have the same non-zero diagonal entries, hence $d_r\neq \pm 1$.


Conversely, assume that $G$ is circuit-coprime. After fixing some spanning tree $T$, consider the lfc-matrix $N_{(G,l)}$. We only need to prove that the diagonal entries of the Smith normal form of $N_{(G,l)}$ are all $1$, which amounts to proving that the greatest common divisor $d$ of the minors of order $r$ of the lfc-matrix $N_{(G,l)}$ is $1$.

As we have seen in  \ref{sectioncircuitmatrices}, we have the relation
$$N_{(G,l)}=N_G\cdot L.$$
Let $N'$ be a maximal square submatrix of $N_{(G,l)}$. Then $N'$ corresponds to $r$ edges of $G$, which we denote $e_{i_1},e_{i_2},\ldots,e_{i_r}$. Let $N''$ be the corresponding square submatrix of $N_G$. We have the relation
$$\det N'=\prod_{j=1}^rl(e_{i_j})\det N''$$ 
By \citep{TS}, Theorem 6.15, all minors of $N_G$ are either $1,0$ or $-1$, hence $\det N''$ is either $1,0$ or $-1$. Moreover, by \citep{TS}, Theorem 6.10, a square submatrix of order $r$ of $N_G$ has determinant $\pm 1$ if and only if the corresponding $r$ edges are the complement of a spanning tree. Hence $\det N'=\pm\prod_{i=1}^rl(e_{i_j})$ if the edges $e_{i_1},e_{i_2},\ldots,e_{i_r}$ form the complement of a spanning tree of $G$, otherwise $\det N'=0$. We claim that 
$$d:=\gcd\{\det N'| N' \mbox{ is an } r\times r \mbox{ square submatrix of } N_{(G,l)}\}=1.$$
Let $p$ be a prime number and denote by $E_p$ the set of edges $e$ of $G$ whose label $l(e)$ is divisible by $p$. Because $(G,l)$ is circuit-coprime, $E_p$ contains no circuit; hence $E_p$ is contained in some spanning tree $T$ of $G$. There are exactly $r$ edges, $e_1,e_2,\ldots, e_r$, that do not belong to $T$. These give a square $r\times r$ submatrix of $N_{(G,l)}$ whose determinant is $\prod_{i=1}^r l(e_i)\not\equiv 0$ (mod $p$), since $e_1,\ldots, e_r\not\in E_p$. Hence $p\nmid d$. It follows that $d=1$; since $d_i|d_{i+1}$ for all $i=1,\ldots, r-1$ and $d_r|d$, we obtain the result. 
\end{proof}

\begin{proposition}
Let $(G,l)=(V,E,l)$ be an $\N$-labelled graph. The group homomorphism $\overline{\epsilon}\colon H\rightarrow \til H$ is an isomorphism if and only if $(G,l)$ is circuit-coprime.
\end{proposition}

\begin{proof}
We already know that $\overline{\epsilon}\colon H\ra \til H$ is injective by Lemma \ref{injection}. It is surjective if and only if for every vertex-labelling $\alpha\in\Z^{\til V}$, there exists $\til \phi\in \til{\mathcal C}$ such that $\til\delta(\til\phi)+\alpha$ is in the image of the extension-by-zero map $\epsilon\colon \Z^V\ra \Z^{\til V}$, i.e. $\til\delta(\til\phi)+\alpha$ is supported on the set of old vertices. We may of course assume that $\alpha$ belongs to the canonical basis of $\Z^{\til V}$. That is, $\alpha=\chi_v$ for some vertex $v$ of $\til G$, where
$$\chi_v(w)=\begin{cases}
1 & \mbox{ if } w=v \\
0 & \mbox{ if } w\neq v.
\end{cases}
$$
If $v$ is an old vertex of $\til G$, $\chi_v$ is an extension by zero of a vertex-labelling on $G$, so we may assume that $v$ is a new vertex. Then $v$ belongs to some path $P\subset \til G$ associated to some edge $\overline e\in E$. Denote by $w_0,w_1,\ldots,w_{l(\overline e)}$ the vertices of the path $P$, so that $w_0$ and $w_{l(\overline e)}$ are old vertices, and the numbering of the indices follows the order of the vertices on the path. For every $i=1,\ldots,l(\overline e)$, let $\alpha_i=\chi_{w_i}-\chi_{w_0}\in\Z^{\til V}$ be the vertex-labelling that has value $1$ at $w_i$, value $-1$ at $w_0$, and value $0$ everywhere else. Then it is easy to check that the images $\overline{\alpha}_i$ of the $\alpha_i$ in $\til H$ satisfy $k\overline{\alpha}_1=\overline{\alpha}_k$ for all $k=1,\ldots,l(\overline e)$. Hence, if $\overline{\alpha_1}$ is in the image of $\overline{\epsilon}\colon H\ra \til H$, so are all the $\overline{\alpha_i}$ for $1\leq i \leq l(e)$.
This shows that we can take $v$ to be equal to $w_1$; hence $\chi_{v}=\chi_{w_1}$ takes value $1$ on a new vertex $v$ adjacent to an old vertex, and value zero at all other vertices. 

We ask whether an element $\til \varphi\in \til{\mathcal C}$ exists such that $\til \delta(\til \varphi)+\chi_{w_1}$ is supported only on the old vertices. In other words, $\til\delta(\til\varphi)$ must be zero on all new vertices except for the vertex $w_1$, where it has to take the value $-1$. This is equivalent to asking that, for every new vertex $z$, adjacent \begin{equation}
\begin{cases}
\til\varphi(z)-\til\varphi(z_1)=\til\varphi(z_2)-\til\varphi(z) & \mbox{ if } z\neq w_1\\ 
(\til\varphi(z_1)-\til\varphi(z))+(\til\varphi(z_2)-\til\varphi(z))= -1 & \mbox{ if } z=w_1
\end{cases}
\end{equation}
holds. 

We claim that such a $\til\varphi$ exists if and only if there exists a vertex-labelling $\varphi$ of the graph $G$, such that, for every edge $e\in E$ with endpoints $v_0,v_1$,

\begin{equation}
\begin{cases} \label{AAA}
\varphi(v_1)-\varphi(v_0) \equiv 0 \mbox{ mod } l(e) & \mbox{ if }
e\neq\overline e \\
\varphi(v_1)-\varphi(v_0) \equiv 1 \mbox{ mod } l(e) & \mbox{ if } e=\overline e, v_0=w_0, v_1=w_{l(e)}
\end{cases}
\end{equation}
where we have identified the old vertices $w_0,w_{l(e)}$ with the corresponding vertices in $G$.
Indeed, given $\til \varphi$ one obtains $\phi$ simply by restriction to old vertices.
Conversely, given a $\phi$ as in \eqref{AAA}, $\til \phi$ is obtained as follows: for an edge $e\neq \overline e$, we define $\til \phi$ on the corresponding path $\{ z_0=v_0, z_1, z_2, \ldots, z_{l(e)}=v_1\}$ by:
$$\forall k=0,1,\ldots,l(e), \;\;\; \til \phi (z_k)=\frac{k\phi(v_1)+(l(e)-k)\phi(v_0)}{l(e)}.$$
On the path $\{w_0=v_0,w_1,\ldots,w_{l(\overline e)}=v_1\}$ corresponding to the edge $\overline e$, we set instead
$$ \til \phi (w_k)=\begin{cases}\frac{k\phi(v_1)+(l(e)-k)(\phi(v_0)+1)}{l(e)} & \mbox{ if }  k\in\{1,2,\ldots,l(\overline e)\}\\
\til \phi(v_0)  & \mbox{ if }  k=0;
\end{cases}
$$
which establishes the claim.

If the graph $G$ is a tree it is clear that such a $\varphi$ can be found. If there are circuits in $G$, the existence of a solution $\varphi$ depends of course on the labels of the circuits. Fix an orientation on $G$, so that we have source and target functions $s,t\colon E\rightarrow V$, and so that $s(\overline e)=w_0$, $t(\overline e)=w_{l(\overline e)}$. Assume that a vertex-labelling $\varphi$ of $G$ satisfying the conditions \eqref{AAA} exists. In particular we have that $\varphi(t(\overline{e}))-\varphi(s(\overline e))\equiv 1 \mbox{ mod } l(\overline e)$. For every edge $e\in E$ let 
$$
x(e):=\begin{cases}
\frac{\varphi(t(e))-\varphi(s(e))}{l(e)}  & \mbox{ if } e\neq \overline e\\
\frac{\varphi(t(e))-\varphi(s(e))-1}{l(e)} & \mbox{ if } e=\overline e
\end{cases}
$$

Let $C\subset G$ be a circuit consisting of 
vertices $v_0,v_1,\ldots,v_s=v_0$ connected by edges $e_0,e_1,e_2,\ldots,e_s=e_0$, so that $e_i$ connects $v_i$ and $v_{i+1}$ for every $i\in \Z/s\Z$. Notice that the increasing numbering gives an orientation to $C$. We have 
$$(\varphi(v_s)-\varphi(v_{s-1}))+(\varphi(v_{s-1})-\varphi(v_{s-2}))+\ldots + (\varphi(v_1)-\varphi(v_s))=0.$$
Setting 
\begin{eqnarray}
a_i=\begin{cases}
1 & \mbox{ if } t(e_i)=v_{i+1}, s(e_i)=v_i \\
-1 & \mbox{ if } t(e_i)=v_{i}, s(e_i)=v_{i+1}
\end{cases}
\end{eqnarray}
for every $i\in \Z/s\Z$, we obtain $$\sum a_ix_{e_i}l(e_i)=0$$
if the edge $\overline e$ does not belong to the circuit $C$, whereas if $\overline e \in C$ we have
$$\sum a_ix_{e_i}l(e_i)=\begin{cases}
-1 & \mbox{ if the orientations of }C\mbox{ and }\overline e\mbox{ agree;}\\
1 & \mbox{ if the orientations of }C\mbox{ and }\overline e\mbox{ do not agree;}

\end{cases}$$

Let $C_1,\ldots,C_m$ be the circuits of $G$. Choose an orientation for each circuit, so that we can form the labelled circuit matrix $M_{(G,l)}$ associated to $G$. We see that the vector $\underline x=(x_1,\ldots,x_n)$ is a solution of 
$$M_{(G,l)} x= b(\overline e)$$
where $ b(\overline e)=(b_1,\ldots,b_m)$ with
$$
b_i=\begin{cases}
0 & \mbox{ if } \overline e\not\in C_i; \\
-1 & \mbox{ if } \overline e\in C_i \mbox{ and the orientation of } \overline e \mbox{ agrees with the } \\
& \mbox{ orientation of } C_i; \\
1 & \mbox{ if } \overline e\in C_i \mbox{ and the orientation of } \overline e \mbox{ does not agree with the } \\
& \mbox{ orientation of } C_i.
\end{cases}
$$

Conversely, a solution $ x\in \Z^n$ to the system $M_{(G,l)} x= b(\overline e)$ yields a vertex labelling $\varphi$ as in \eqref{AAA}. We conclude that the map $\overline\epsilon\colon H\rightarrow \til H$ is surjective if and only if for every edge $e\in E$, there is a solution $ x\in \Z^n$ to 
$$M_{(G,l)} x=b(e).$$
After having chosen a spanning tree $T$ and formed the lfc-matrix $N_{(G,l)}$, this is in turn equivalent to the map $\Z^n\ra \Z^r$ defined by $N_{(G,l)}$ being surjective. Indeed, the set $\{b(e)|e \mbox{ is a link of }T\}$ is a basis for $\Z^r$. Now, $N_{(G,l)}$ is surjective if and only if its Smith normal form (or equivalently the one of $M_{(G,l)}$) has only $1$'s on the diagonal. By Lemma \ref{SNFcoprime}, we conclude.

\end{proof}

\subsection{\texorpdfstring{$\N_{\infty}$}{N}-labelled graphs} 
We want to generalize the results of the previous subsection to labelled graphs whose labels can attain the value $\infty$. Denote by $\N_{\infty}$ the set $\Z_{\geq 1}\cup \{\infty\}$. Let $(G,l)=(V,E,l)$ be the datum of a graph, with set of vertices $V$ and set of edges $E$, and of a function $l\colon E\rightarrow \N_{\infty}$. We say that $(G,l)$ is an \textit{$\N_{\infty}$-labelled graph}. 

The notions of Cartier vertex labelling \ref{cartierlabel} and multidegree operator \ref{defdelta} carry over to this setting without change, imposing that the only integer divisible by $\infty$ is $0$, and setting $\frac{0}{\infty}=0$ in the definition of multidegree operator. In particular, if a vertex-labelling on $(G,l)$ is Cartier, it attains the same value at the two extremal vertices of an edge with label $\infty$.

\begin{definition} \label{blowupgraph}
Given an $\N_{\infty}$-labelled graph $(G,l)=(V,E,l)$ we define the \textit{first-blow-up graph} $G_1=(V_1,E_1,l_1)$ to be the $\N_{\infty}$-labelled graph constructed as follows starting from $(G,l)$: every edge $e\in E$ with $l(e)=1$ is preserved unaltered; every edge $e\in E$ with $l(e)\geq 2$ is replaced by a path consisting of an edge labelled by $1$, followed by an edge labelled by $l(e)-2$ (which could equal $0$ or $\infty$), followed by an edge labelled by $1$.

We define inductively for every integer $n\geq 1$ the \textit{$n$-th blow-up graph} $G_n=(V_n,E_n,l_n)$ as the first-blow-up graph of $G_{n-1}$.
\end{definition}

\begin{example}
Figure \ref{blowupgraphs} shows an $\mathbb N_{\infty}$-labelled graph (\subref{0th}) with its first (\subref{1st}) and second (\subref{2nd}) blow-up graphs. 
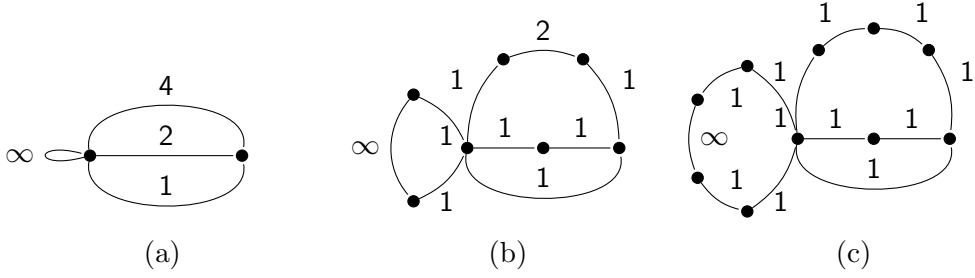
\begin{figure}[h]
 \centering
 \begin{subfigure}[b]{0.3\textwidth}
\begin{tikzpicture}[>=stealth',shorten >=1pt,auto,node distance=2cm,main node/.style={circle,draw,fill=black,inner sep=1.5,minimum size=0.4pt},every loop/.style={min distance=7mm}]

  \node[main node] (1) {};
  \node[main node] (2) [right of=1] {};

  \path[every node/.style={font=\sffamily\small}]    
      (1) edge node {2} (2)
      edge [loop left] node {$\infty $} (1)  
      edge [bend right=100]  node {1} (2)
      edge [bend left=100] node  {4} (2)
       ;
\end{tikzpicture}

\caption{\label{0th}}
\end{subfigure}
\begin{subfigure}[b]{0.3\textwidth}
\begin{tikzpicture}[>=stealth',shorten >=1pt,auto,node distance=1cm,main node/.style={circle,draw,fill=black,inner sep=1.5,minimum size=0.4pt},every loop/.style={min distance=7mm}]

  \node[main node] (2)  {};
  \node[main node] (0) [above left of =2] {};
  \node[main node] (1) [below left of =2] {};
  
  \node[main node] (3) [right of=2] {};
  \node[main node] (4) [right of=3] {};
  \node[main node] (5) [above right=1.5cm of 2, xshift=-0.7cm] {};
  \node[main node] (6) [above left=1.5cm of 4, xshift=0.7cm] {};

  \path[every node/.style={font=\sffamily\small}]

  (0) edge [bend left=20] node [below] {1} (2)
  (0) edge [bend right=40] node [left] {$\infty$} (1)
  (1) edge [bend right=20] node [below] {1} (2)
  (2) edge [bend right=100] node {1} (4)
  (2) edge node {1} (3)
  (2) edge [bend left=20] node {1} (5)
  (3) edge node {1} (4)
  (5) edge [bend left=20] node {2} (6)
  (6) edge [bend left=20] node {1} (4)
  ;
\end{tikzpicture}
\caption{\label{1st}}
\end{subfigure}
\begin{subfigure}[b]{0.3\textwidth}
\begin{tikzpicture}[>=stealth',shorten >=1pt,auto,node distance=1cm,main node/.style={circle,draw,fill=black,inner sep=1.5,minimum size=0.4pt},every loop/.style={min distance=7mm}]

  \node[main node] (b) [above left=1.7cm of 2, yshift=-0.8cm] {};
  \node[main node] (c) [below left=1.7cm of 2, yshift=0.8cm] {};
  \node[main node] (2)  {};
  \node[main node] (0) [above left=1.2cm of 2, xshift=0.3cm] {};
  \node[main node] (1) [below left=1.2cm of 2, xshift=0.3cm] {};
  
  \node[main node] (3) [right of=2] {};
  \node[main node] (4) [right of=3] {};
  \node[main node] (5) [above right=1.5cm of 2, xshift=-0.9cm] {};
  \node[main node] (a) [above=1.3cm of 3] {};
  \node[main node] (6) [above left=1.5cm of 4, xshift=0.9cm] {};

  \path[every node/.style={font=\sffamily\small}]
  
  (0) edge [bend right=20] node {1} (b)
  (c) edge [bend right=20] node {1} (1)
  (b) edge [bend right=20] node {$\infty $} (c)
  (0) edge [bend left=20] node [below] {1} (2)
  (1) edge [bend right=20] node [below] {1} (2)
  (2) edge [bend right=100] node {1} (4)
  (2) edge node {1} (3)
  (2) edge [bend left=20] node {1} (5)
  (3) edge node {1} (4)
  (5) edge [bend left=20] node {1} (a)
  (a) edge [bend left=20] node {1} (6)
  (6) edge [bend left=20] node {1} (4)
  ;
\end{tikzpicture}
\caption{\label{2nd}}
\end{subfigure}
\caption{An $\mathbb N_{\infty}$-labelled graph (\subref{0th}) with its first (\subref{1st}) and second (\subref{2nd}) blow-up graphs \label{blowupgraphs}}
\end{figure}

\end{example}

Denote by $\mathcal C_n$ the group of Cartier vertex-labellings on $(G_n,l_n)$. Just as in \eqref{diagramCartier}, we obtain a commutative diagram
$$
\centering
\xymatrix{
\mathcal C \ar[r]^{\delta} \ar[d]^{\iota_1} & \Z^V \ar[d]^{\epsilon_1} \\
\mathcal C_1 \ar[r]^{\delta_1} \ar[d]^{\iota_2} & \Z^{V_1} \ar[d]^{\epsilon_2} \\
\vdots  \ar[d] & \vdots \ar[d] \\
\mathcal C_n \ar[r]^{\delta_n} \ar[d]^{\iota_n} & \Z^{V_n} \ar[d]^{\epsilon_n} \\
\vdots & \vdots 
}
$$
The vertical maps $\epsilon_j$ are once again extension by zero; the maps $\iota_j$ are defined as follows: if $e$ is an edge of $G_{j-1}$ which is replaced in $G_j$ by a path consisting of vertices $v_0=v, v_1,v_2,v_3=w$ (with possibly $v_1=v_2$, if $l_{j-1}(e)=2$), and $\phi$ is Cartier vertex labelling on $G_{j-1}$, we set $\iota_j(\phi)$ to take the value $\phi(v)$ at $v_0$, $\frac{(l(e)-1)\phi(v)+\phi(w)}{l(e)}$ at $v_1$, $\frac{\phi(v)+(l(e)-1)\phi(w)}{l(e)}$ at $v_2$, $\phi(w)$ at  $v_3$. 
The diagram above gives rise to a chain of group homomorphisms
\begin{equation} \label{chainH}
H\ra H_1\ra H_2\ra \ldots \ra H_n\ra \ldots
\end{equation}
between the cokernels of the rows. Each map of the chain \eqref{chainH} is injective; we ask whether they are all isomorphisms, i.e. under which conditions
\begin{equation} \label{HH}
H\ra \colim H_i
\end{equation}
is an isomorphism.
\begin{definition} \label{circuitcoprimeinfty}
Let $(G,l)=(V,E,l)$ be an $\N_{\infty}$-labelled graph. Let $ (G_c,l_c)=(V_c, E_c,  l_c)$ be the labelled graph obtained by contracting all edges $e$ with $l(e)=\infty$. 

We say that $(G,l)$ is \textit{circuit-coprime} if $(G_c,l_c)$ is circuit-coprime as in Definition \ref{defcc}.
\end{definition}

\begin{proposition} \label{inftygraphs}
Let $(G,l)$ be an $\N_{\infty}$-labelled graph. The map \eqref{HH} is an isomorphism if and only if $(G,l)$ is circuit-coprime.
\end{proposition}

\begin{proof}
Let $(G_c,l_c)$ be as in Definition \ref{circuitcoprimeinfty}. Denote by $(\widetilde{G}, \til l)$ the labelled graph obtained from $(G,l)$ by replacing every edge $e$ such that $l(e)<\infty$ by a path of length $l(e)$ and by leaving unaltered the edges with label $\infty$. Contracting the edges of $(\til G,\til l)$ with label $\infty$, we recover the total-blow up graph $(\til G_c, \til l_c)$ of $(G_c,l_c)$. As usual, we have a diagram of exact sequences 
$$
\xymatrix{
0 \ar[r] & \Z \ar[r]\ar[d] & \mathcal C \ar[r]^{\delta}\ar[d]^{\iota} & \Z^V \ar[r]\ar[d]^{\epsilon} & H \ar[r]\ar[d] & 0 \\
0 \ar[r] & \Z \ar[r] & \til{\mathcal{C}} \ar[r]^{\til \delta} & \Z^{\til V} \ar[r] & \til H \ar[r] & 0 \\
}
$$
We claim that $H\ra \til H$ is an isomorphism if and only if $(G,l)$ is circuit coprime. 

The map $V\ra V_c$ induces a map between vertex-labellings $\Z^{V_c}\ra \Z^V$, which moreover sends the group of Cartier vertex-labellings of $(G_c,l_c)$ isomorphically to the group of Cartier-vertex labellings of $(G,l)$. We obtain a diagram

$$
\xymatrix{
0 \ar[r] &  \mathcal C_c/\Z \ar[r]\ar[d] \ar@{-->}[rdd]^(.35){\alpha} & \Z^{V_c} \ar[r]\ar[d]\ar@{-->}[rdd]^(.35){\beta} & H_c \ar[r]\ar[d]\ar@{-->}[rdd] & 0 & \\
0 \ar[r] &  \til{\mathcal{C}}_c/\Z \ar[r]\ar@{-->}[rdd]^(.35){\gamma} & \Z^{\til V_c} \ar[r]\ar@{-->}[rdd]^(.35){\delta} & \til H_c \ar[r]\ar@{-->}[rdd] & 0 & \\
& 0 \ar[r] &  \mathcal C/\Z \ar[r]\ar[d] & \Z^{V} \ar[r]\ar[d] & H \ar[r]\ar[d] & 0 \\
& 0 \ar[r] &  \til{\mathcal{C}}/\Z \ar[r] & \Z^{\til V} \ar[r] & \til H \ar[r] & 0 \\
}
$$
where all parallelograms are commutative.
Let
$$
\xymatrix{
D_1 \ar[r] \ar@{-->}[rd] & D_2 \ar[r]\ar@{-->}[rd] & D_3 \ar@{-->}[rd]\ar[r] & 0 & \\
& E_1 \ar[r]  & E_2 \ar[r] & E_3 \ar[r] & 0 \\
}
$$
be the induced diagram between the cokernels of the vertical maps. Because the maps $H\ra \til H$ and $ H_c\ra \til H_c$ are injective, by the Snake lemma the maps $D_1\ra D_2$ and $E_1\ra E_2$ are in fact injective. The maps $\alpha$ and $\gamma$ are isomorphisms, hence $D_1\ra E_1$ is an isomorphism. Moreover $D_2\cong \Z^{\til V_c-V_c}\cong \Z^{\til V-V}  \cong E_2$, and the map between them is easily seen to be an isomorphism. Hence $D_3\ra E_3$ is an isomorphism. Now, $H\ra \til H$ is an isomorphism if and only if $E_3$ is zero, if and only if $D_3$ is zero, if and only if $(G_c,l_c)$ is circuit-coprime, establishing the claim.

For $n$ big enough, $G_n$ is a blow-up of $\til G$, hence the map $H\ra H_n$ factors as $H \ra \til H \ra H_n$, with $\til H\ra H_n$ injective. If we show that the map $\til H \ra H_n$ is also surjective, we win. Notice that the labelled graph $G_n$ is obtained from $\til G$ replacing each edge labelled by $\infty$ by a path of length $2n+1$, with a central edge labelled by $\infty$ and all other edges labelled by $1$. Take one such path and denote by $w_0,w_1,\ldots,w_{2n},w_{2n+1}$ its vertices. The edge between $w_n$ and $w_{n+1}$ is labelled by $\infty$. As usual, we call old vertices the vertices in the image of the map $\til V\ra V_n$, so $w_0$ and $w_{2n+1}$ are old vertices.  Let $\chi_{w_1}\in \Z^{V_n}$ be the vertex-labelling taking value $1$ at $w_1$ and zero everywhere else. Then $\delta(\chi_{w_0})+\chi_{w_1}$ is supported on the old vertices, hence $\chi_{w_1}$ is in the image of the map $\til H\ra H_n$. Also, in $H_n$ we have $\chi_{w_k}=k\chi_{w_1}$ for $i=1,\ldots,n$, so also these vertex-labellings are in the image of $\til H\ra H_n$. A similar argument works for $\chi_{w_k}$ with $k=n+1,\ldots,2n$. Hence the map $\til H\ra H_n$ is surjective. 

\end{proof}

\section{Semi-factoriality of nodal curves}\label{s6}
Let $S$ be the spectrum of a discrete valuation ring $R$ having fraction field $K$ , residue field $k$ and uniformizer $t$. Let $f\colon \X\ra S$ be a nodal curve whose special fibre has split singularities, and $\Gamma=(V,E)$ be the dual graph of the special fibre $\X_k$. For any $v\in V$, we denote by $X_v$ the corresponding irreducible component of the special fibre $\X_k$.

\begin{definition}
 The \textit{labelled graph} of $\X\ra S$ is the $\N_{\infty}$-labelled graph $(\Gamma,l)$ whose labelling $l$ assigns to each edge of $\Gamma$ the thickness (see \ref{subsectionnodalcurves}) of the corresponding singular point of $\X_{k}$.
\end{definition}

Our aim is to relate the property of being circuit-coprime for the graph $(\Gamma,l)$ to the semi-factoriality of $f\colon \X\ra S$. To this end, we are going to provide a dictionary between the geometry of $\X/S$ and the combinatorial objects introduced in \ref{sectiongraphtheory}.

Denote by $\Div_k(\X)$ the group of Weil divisors on $\X$ supported on the special fibre $\X_k$. It is the free abelian group generated by the irreducible components of $\X_k$.  Hence we obtain a natural isomorphism $\Div_k(\X)\ra \Z^V$. 

Let $\mathcal C(\X)$ be the group of Cartier divisors on $\X$ whose restriction to the generic fibre $\X_K$ is trivial. We claim that the natural map $\mathcal C(\X)\ra \Div_k(\X)$ is injective. This follows from (\citep{EGA}, 21.6.9 (i)) under the assumption that $\X$ is normal, which is not satisfied if $\X/S$ has singular generic fibre. However, the proof only requires that for all $x\in \X_k$, $\depth(\O_{\X,x})=1$ implies $\dim \O_{\X,x}=1$. This is immediately checked: let $x\in \X_k$ with $\dim \O_{\X,x}\neq 1$; then $x$ is a closed point of $\X_k$. By $S$-flatness of $\X$, the uniformizer $t$ is not a zero divisor in $\O_{\X,x}$; as $\X_k$ is reduced, $\O_{\X,x}/t\O_{\X,x}$ is reduced. Every reduced noetherian ring of dimension $1$ is Cohen-Macaulay, hence $\depth(\O_{\X,x}/t\O_{\X,x})=1$, and we deduce by \citep{stacks}\href{http://stacks.math.columbia.edu/tag/0AUI}{TAG 0AUI} that $\depth(\O_{\X,x})=2$, establishing the claim. Hence $\mathcal C(\X)$ is in a natural way a subgroup of $\Div_k(\X)$.

Finally, denote by $E(\X)$ the kernel of the restriction map $\Pic(\X)\rightarrow \Pic(\X_K)$, so that $E(\X)$ is the group of isomorphism classes of line bundles on $\X$ that are generically trivial. We have an exact sequence of groups $$0\ra \Z \ra \mathcal C(\X)\rightarrow E(\X) \ra 0$$
where the first map sends $1$ to $\divv(t)$ and the second map sends $D$ to $\O_{\X}(D)$.
Indeed, every principal Cartier divisor supported on the special fibre belongs to $\mathbb Z\divv(t)$. For this we can reduce to showing that every regular function on $\X$ that is generically invertible is of the form $t^nu$ for some $n\in \mathbb Z_{\geq 0}$ and $u\in \O_{\X}(\X)^{\times}$. By \citep{stacks}\href{http://stacks.math.columbia.edu/tag/0AY8}{TAG 0AY8} we have $f_*\O_{\X}=\O_S$, from which the claim easily follows.
\begin{lemma} \label{comparison}
Hypotheses as in the beginning of this section.
\begin{itemize}
\item[i)] The natural isomorphism $\Div_k(\X)\ra \Z^V$ identifies $\mathcal C(\X)\subset \Div_k(\X)$ with the subgroup $\mathcal C\subset \Z^V$ of Cartier vertex labellings (Definition \ref{cartierlabel}).
\end{itemize}

Let $$0\ra \Z\ra \mathcal C\xrightarrow{\delta} \Z^V$$ be the exact sequence of Lemma \ref{kernelZ}, where $\delta$ is the multi-degree operator (Definition \ref{defdelta}).
\begin{itemize}
\item[ii)] The isomorphism $\mathcal C(\X)\ra \mathcal C$ induces an exact sequence
$$0\ra \Z\ra \mathcal C(\X)\xrightarrow{\delta_{\X}} \Z^V.$$
The first arrow is the map $1\mapsto \divv(t)$; the map $\delta_{\X}$ factors via the map $E(\X)\rightarrow \Z^V$, which sends a line bundle $\mathcal L$ to the vertex labelling 

$$v\mapsto \deg \mathcal L_{|X_v}.$$ 
\end{itemize}

Let 
$$\ldots\ra \X_n \ra \ldots \ra \X_1\ra \X_0=\X$$
be the chain of blowing-ups \eqref{chainn}. Denote by $\pi_n$ the composition $\X_n\ra \X$.
\begin{itemize}
\item[iii)] For every $n\geq 0$ the labelled graph of $\X_n\ra S$ is the $n$-th blow-up graph $(\Gamma_n,l_n)$ of $(\Gamma,l)$ (Definition \ref{blowupgraph}). The new vertices of $(\Gamma_n,l_n)$ correspond to the irreducible components of the exceptional fibre of $\X_n\ra \X$.

\item[iv)] Let $\mathcal C_n$ be the group of Cartier vertex labellings on $\X_n$. The map $\mathcal C(\X)\ra \mathcal C(\X_n)$ induced by $\iota\colon\mathcal C\ra \mathcal C_n$ (\ref{subsCLBG}) descends to the pullback map $\pi^*_n\colon E(\X)\rightarrow E(\X_n)$. 
\end{itemize}
\end{lemma}

\begin{proof}\leavevmode
\begin{itemize}

\item[i)] 
Let $D=\sum_v n_vX_v \in \Div_k(\X)$. We want to show the equivalence of the two conditions:
\begin{itemize}

\item[a)] for every node $p\in \X_k$ lying on distinct components $X_w,X_z$ of $\X_k$, the thickness $\tau_p$ divides $n_w-n_z$ (with the convention that $\infty$ divides only $0$);
\item[b)] $D$ is Cartier.
\end{itemize}  
As every Weil divisor $D$ is Cartier on the generic fibre and on the regular locus of $\X$, we may fix a node $p\in \X_k$ and reduce to work on the complete local ring $\widehat{\O}_{\X,p}$. We identify $\widehat{\O}_{\X,p}$ with $A=\widehat{R}[[x,y]]/xy-t^{\tau_p}.$ Let $X_w$ and $X_z$ be the components of $\X_k$ through $p$, and let $Y_w$, $Y_z$ be their preimages in $\Spec A$, which are given by the ideals $(x,t)$ and $(y,t)$ of $A$ respectively.

Assume a) is true; we are going to deduce that $D$ is Cartier at $p$. We may assume that the two components $X_w$ and $X_z$ are distinct, otherwise $D$ is given by $\divv(t^{n_w})$ locally at $p$ and is automatically Cartier at $p$. As $\divv(x)=\tau_pY_w$, we have that $(n_w-n_z)Y_w=\divv(x^{\frac{n_w-n_z}{\tau_p}})$ is Cartier. Therefore $D-\divv(t^{n_z})=\sum_v(n_v-n_z)X_v$ is Cartier at $p$, and also $D$ is.

Assume now b) and that $p$ lies on distinct components $X_w,X_z$ of $\X_k$. We may assume that the restriction of $D$ to $\Spec A$, $n_wY_w+n_zY_z$, is the divisor of some regular function $f\in A=\widehat{R}[[x,y]]/xy-t^{\tau_p}$. We first consider the case $\tau_p=\infty$. As $f$ is a unit in $A[t^{-1}]$, there exists $g\in A$ and $n\geq 0$ such that $fg=t^n$. Now, let $f_x$ be the image of $f$ in $A/xA$. As the latter is a unique factorization domain, $f_x=t^{m_1}u_1$ for some unit $u_1\in (A/xA)^{\times}$ and $m_1\leq n$. Moreover, we have $m_1=n_w$. Similarly, we write $f_y=t^{m_2}u_2\in A/yA$, with $m_2=n_z$. As the images of $f_x$ and $f_y$ in $A/(x,y)A=R$ coincide, we find that $m_1=m_2$, that is, $n_w=n_z$, as desired. Now we remain with the case $\tau_p\neq \infty$. Replacing $f$ by $ft^{-n_z}$, we get $\divv(f)=(n_w-n_z)Y_w$. We want to show that $\tau_p$ divides $m:=n_w-n_z$. Let $d=\gcd(m,\tau_p)$. As $\divv(x)=\tau_pY_w$, we may replace $f$ by a product of powers of $f$ and $x$ and assume that $m=d$. Write $\tau_p=m\alpha$, for some $\alpha\in\Z$. We have $\divv(f^{\alpha}/x)=0$, hence, as $\Spec A$ is normal, $f^{\alpha}/x$ is a unit in $A$. Now, reducing modulo $t$, one can easily see that $\alpha$ has to be $1$, so $m=\tau_p$ as desired.
\item[ii)] The composition $\Z\ra \mathcal C\ra \mathcal C(\X)$ sends $1$ to $\sum_vX_v=\X_k=\divv(t).$ The map $\delta_{\X}$ factors via the cokernel of $\Z\ra \mathcal C(\X)$, which is indeed $E(\X)$. For the characterization of the map $\delta_{\X}$, recall first that $\delta\colon\mathcal C\rightarrow \Z^V$ sends a Cartier vertex labelling $\varphi$ to the vertex labelling

$$
v\mapsto \sum_{\substack{\mbox{\scriptsize edges }  e \\ \mbox{\scriptsize  incident to }v }} \frac{\phi(w)-\phi(v)}{l(e)} 
$$
where $w$ denotes the other endpoint of $e$.
The composition $\delta_{\X}\colon\mathcal C(\X)\ra \mathcal C\ra \Z^{V}$ sends a Cartier divisor $D=\sum_vn_vX_v$ to 
$$
v\mapsto \sum_{\substack{\mbox{\scriptsize nodes }  p \\ \mbox{\scriptsize  lying on }X_v }} \frac{n_w-n_v}{\tau_p} 
$$
with $\tau_p$ being the thickness of the node $p$, $X_w$ the second component passing through $p$. We want to check that $\delta_{\X}(D)$ is the vertex labelling $v\mapsto \deg \O(D)_{|X_v}$. Fix a vertex $z$; multiplication by $t^{n_z}$ gives an isomorphism $\O(D)\cong \O(D')$ where $D'=\sum_v(n_v-n_z)X_v$. We reduce to computing the contribution to $\deg \O(D')_{|X_z}$ coming from $(n_v-n_z)X_v$, where $v\in V$ is some vertex different from $z$. The contribution is zero if $\X_v$ and $\X_z$ do not intersect; otherwise, let $p\in X_v\cap X_z$, with thickness $\tau_p$. Notice that $\tau_p|n_v-n_z$. Locally at $p$, the divisor $(n_v-n_z)X_v$ is given by the fractional ideal $I=(x^{(n_v-n_z)/\tau_p},t^{n_v-n_z})=(x^{(n_v-n_z)/\tau_p})$ of $\widehat \O_{\X,p}\cong\widehat R[[x,y]]/xy-t^{\tau_p}$. Restricting to the branch $y=0,t=0$, we obtain the fractional ideal $I\otimes \widehat \O_{\X,p}/y=(x^{(n_v-n_z)/\tau_p})$ of $k[[x]]$, hence a contribution of $(n_v-n_z)/\tau_p$ to the degree of $\O(D')_{|X_z}$. Summing over all the nodes in $X_v\cap X_z$, we recover the map $\delta_{\X}$.
 
\item[iii)] This can be read directly in the description of the effect of blowing-up on the special fibre provided in \ref{blup}.
\item[iv)] The commutative diagram 
$$
\xymatrix{
0\ar[r] & \Z \ar[r]\ar[d]^{\id} & \mathcal C(\X) \ar[r]\ar[d]^{\iota} & E(\X) \ar[r]\ar@{-->}[d]^{\overline \iota} & 0 \\
0\ar[r] & \Z \ar[r] & \mathcal C(\X_n) \ar[r] & E(\X_n) \ar[r] & 0 \\
}$$
yields a map $\overline{\iota}\colon E(\X)\ra E(\X_n).$ Such map fits into the commutative diagram 
$$
\xymatrix{
 E(\X) \ar[r]^{\delta_{\X}}\ar[d]^{\overline{\iota}} & \Z^{V} \ar[d]^{\epsilon} \\
E(\X_n) \ar[r]^{\delta_{\X_n}} & \Z^{V_n} \\
}
$$
where $\epsilon\colon \Z^V\ra \Z^{V_n}$ is the extension by zero map, and the two horizontal maps are induced by the exact sequences as in ii) for $\X$ and $\X_n$. They associate to a line bundle its multi-degree on the special fibre, and are injective. The pullback map $\pi_n^*\colon E(\X)\ra E(\X_n)$ makes the diagram above commutative as well; it follows that it coincides with $\overline{\iota}$. 
\end{itemize}

\end{proof}
\color{black}

\begin{theorem} \label{mainthm}
Let $\X\ra S$ be a nodal curve over a trait whose special fibre has split singularities.
\begin{itemize}
\item[i)]  If the labelled graph $(\Gamma,l)$ is circuit-coprime then $\X\ra S$ is semi-factorial.
\item[ii)] Suppose that $\Gamma(S,\O_S)$ is strictly-henselian. If $\X$ is semi-factorial over $S$, then the labelled graph $(\Gamma,l)$ is circuit-coprime.
\end{itemize}  
\end{theorem}
\begin{proof}
We start with part i). Suppose $\Gamma$ is circuit-coprime. Let $L$ be a line bundle on $\X_K$. By Theorem \ref{extending}, there exists an integer $n\geq 0$ such that $L$ extends to a line bundle $\til{\mathcal L}$ on $\X_n$. Let $(\Gamma_n,l_n)$ be the labelled graph of $\X_n$, which is the $n$-th blow-up graph of $\Gamma$. Denote by $\alpha\in \Z^{V_n}$ the vertex-labelling assigning to each vertex $v$ the degree of the restriction of $\til{\mathcal L}$ to the component of $(\X_n)_k$ corresponding to $v$. By Proposition \ref{inftygraphs}, the map $H\rightarrow H_n$ is an isomorphism; hence there exists a Cartier vertex labelling $\phi$ on $(\Gamma_n,l_n)$ such that $\delta(\phi)+\alpha$ is in the image of the map $\Z^V\rightarrow \Z^{V_n}$. Equivalently (by Lemma \ref{comparison}) there exists a Cartier divisor $D\in \mathcal C(\X_n)$, such that $\delta_{\X_n}(D)+\alpha$ is in the image of $\Z^V\rightarrow \Z^{V_n}$, i.e., $\delta_{\X}(D)+\alpha$ has value zero on all new vertices of $\Gamma_n$. This means precisely that $\O_{\X_n}(D)\otimes \til{\mathcal L}$ has degree zero on every component of the exceptional fibre of $\pi_n\colon\X_n\ra \X$. By Proposition \ref{formalfunction}, $\mathcal L:=(\pi_n)_*(\til{\mathcal L}\otimes \O(D))$ is a line bundle on $\X$, which restricts to $L$ on the generic fibre.

Let's turn to part ii). Suppose that $\Gamma$ is not circuit-coprime. Then there exists $n\geq 0$ such that the map $H\rightarrow H_n$ is not surjective. Let $\alpha$ be a basis element of $\Z^{V_n}$ such that the image of $\alpha$ in $H_n=\Z^{V_n}/\delta_n(\mathcal C_n)$ is not in the image of $H\rightarrow H_n$. Then $\alpha$ takes value $1$ on some vertex $v$ of $\Gamma_n$ and value zero on all other vertices. The vertex $v$ corresponds to an exceptional component $C\cong \mathbb P^1_k$ of $\pi_n\colon \X_n\rightarrow \X$. Let $p$ be a $k$-rational point of $(\X_n)^{sm}_k$ lying on $C$, which exists as $k$ is separably closed. Since the base is henselian, $p$ can be extended to a section $s\colon S\rightarrow \X_n$. The image $D\subset \X_n$ of $s$ defines a Cartier divisor. Let $L:=\O(D)_{|K}$ be its restriction to the generic fibre. Assume by contradiction that $L$ can be extended to a line bundle $\mathcal L$ on $\X$. Then $\mathcal F:=\O(D)\otimes\mathcal \pi_n^*\mathcal L^{-1}$ is generically trivial. Let $D'$ be a Cartier divisor supported on the special fibre of $\X_n$ such that $\O(D')\cong \mathcal F$. Then $D'$ corresponds to a Cartier-vertex labelling $\phi$ of $\Gamma_n$, and $\alpha-\delta_n(\phi)$ is the vertex-labelling associated to the multidegree of $\pi_n^*\mathcal L$. As $\pi^*_n\mathcal L$ has degree zero on every component of the exceptional fibre of $\pi_n\colon\X_n\ra \X$, $\alpha-\delta_n(\phi)$ has value zero on every new vertex of $\Gamma_n$. In particular, $\alpha\delta_n(\varphi)$ is in the image of $H\rightarrow H_n$, and so is $\alpha$, yielding a contradiction.

\end{proof}

\begin{remark}
The assumption that $\Gamma(S,\O_S)$ is strictly-henselian can be replaced by the weaker assumption: for each irreducible component $Y$ of $\X_k$, there exists a line bundle $\mathcal L_Y$ on $\X$ whose restriction to $\X_k$ has degree $1$ on $Y$ and degree $0$ on all other components. 
\end{remark}

\begin{corollary}\label{blowuppp}
Let $\X\ra S$ be a nodal curve over a trait, whose special fibre has split singularities. Let $\pi\colon \widetilde{\X}\ra \X$ be the blowing-up of $\X$ at the finite union of closed points $\X^{nreg}\cap \X_k$. The restriction map
$$\Pic(\widetilde{\X})\ra \Pic(\X_K)$$
is surjective.
\end{corollary}

\begin{proof}
Let $(\Gamma,l)$ be the labelled graph of $\X\ra S$. The labelled graph $(\til \Gamma, \til l)$ of $\widetilde \X\ra S$ is the first-blow-up graph of $\Gamma$ (definition \ref{blowupgraph}). Every edge of $\til \Gamma$ with a label different from $1$ is adjacent to exactly two edges, both with label $1$. Hence $\til \Gamma$ is circuit-coprime, and we conclude by Theorem \ref{mainthm}.

\end{proof}

\begin{corollary}\label{tree}
Let $\X\ra S$ be a nodal curve over a trait, whose special fibre has split singularities. Suppose that the special fibre $\X_k$ is of compact-type (i.e. its dual graph $\Gamma$ is a tree). Then the restriction map
$$\Pic(\X)\ra \Pic(\X_K)$$
is surjective.
\end{corollary}

\begin{proof}
The dual graph $\Gamma$ of the special fibre has no circuits, hence the labelled graph $(\Gamma,l)$ is circuit-coprime.
\end{proof}
In general, semi-factoriality of nodal curves over traits does not descend along \'etale base change, and we cannot drop the assumption in Theorem \ref{mainthm} that the special fibre of the curve has split singularities. Here is an example.

\begin{example}
Let $R=\Q[[t]]$, $K=\Frac R$, $S=\Spec R$, and $$\X=\Proj \frac{R[x,y,z]}{x^2+y^2-t^2z^2}.$$
The curve $\X\ra S$ has smooth generic fibre $\X_K/K$, and a node $P=(t=0,x=0,y=0,z=1)$ on the special fibre.
The section $s\colon S\ra \X$ given by $x=t, y=0, z=1$ goes through the node $P$ . The Cartier divisor on $\X_K$ given by the image of $s_K\colon \Spec K\ra \X_K$ does not extend to a Cartier divisor on $\X$. Indeed, if by contradiction it extended to a Cartier divisor $D$ on $\X$, the difference $D-s$ as Weil divisors would be a Weil divisor supported on the special fibre; hence a Weil divisor linearly equivalent to zero, since the special fibre is irreducible. Then $s$ would be Cartier, which it is not, and we have the contradiction.

On the other hand, the base change of $\X/R$ by the \'etale map $R\ra R':=\Q(i)[[t]]$  is semi-factorial, since its special fibre has split singularities and its graph is a tree. We see that, denoting by $X_1$ and $X_2$ the two components of the special fibre, the Weil divisors $s_{R'}-X_1$ and $s_{R'}-X_2$ are both Cartier, and both extend the Cartier divisor on $\X_{K'}$ given by $s_{K'}$.
\end{example}

\section{Application to N\'eron lft-models of jacobians of nodal curves}\label{s7}

\subsection{Representability of the relative Picard functor}
Let $S$ be a scheme and $\X\ra S$ a curve. We denote by $\Pic_{\X/S}$ the relative Picard functor, that is, the fppf-sheafification of the functor 
\begin{eqnarray}
(\Sch/S)^{opp} & \rightarrow & \Sets \nonumber \\
T & \mapsto & \{\mbox{invertible sheaves on } \X_T\}/\cong \nonumber
\end{eqnarray}

We start with a result on representability of the Picard functor:

\begin{theorem}[\citep{BLR} 9.4/1] \label{algspace}
Let $f\colon\X\ra S$ be a nodal curve. Then the relative Picard functor $\Pic_{\X/S}$ is representable by an algebraic space\footnote{Defined as in \citep{BLR} 8.3/4}, smooth over $S$.
\end{theorem}

\begin{lemma} \label{rigidif}

Let $f\colon\X\rightarrow S$ be a nodal curve admitting a section $s\colon S\rightarrow \X$. Then for any $S$-scheme $T$ the natural map
$$\Pic(\X\times_ST)/\Pic(T)\ra \Pic_{\X/S}(T)$$
is an isomorphism.
\end{lemma}
\begin{proof}
See the discussion about rigidified line bundles on \citep{BLR} 8.1.
\end{proof}

\subsection{N\'eron lft-models}
Let $S$ be a Dedekind scheme, that is, a noetherian normal scheme of dimension $\leq 1$. Then $S$ is a disjoint union of integral Dedekind schemes $S_i$. The \textit{ring of rational functions} of $S$ is the direct sum $K:=\bigoplus_i k(\eta_i)$, where the points $\{\eta_i\}$ are the generic points of the $S_i$. 
\begin{definition}[\citep{BLR}, 10.1/1]
Let $S$ be a Dedekind scheme, with ring of rational functions $K$. Let $A$ be a $K$-scheme. A \textit{N\'eron lft-model} over $S$ for $A$ is the datum of a smooth separated scheme $\mathcal A \ra S$ and a $K$-isomorphism $\varphi\colon \mathcal A\times_SK\rightarrow A$ satisfying the following universal property: for any smooth map of schemes $T\rightarrow S$ and $K$-morphism $f\colon T_{K} \ra A$, there exists a unique $S$-morphism $F\colon T\ra \mathcal A$ with $F_K=f$.

\end{definition}

A N\'eron lft-model differs from a N\'eron model in that the former is not required to be quasi-compact. 

\begin{proposition}[\citep{BLR}, 10.1/2] \label{weakNM}
Let $S$ be a trait and $G$ a smooth separated $S$-group scheme. The following are equivalent:
\begin{itemize}
\item[i)] $G$ is a N\'eron lft-model of its generic fibre;
\item[ii)] for every essentially smooth local extension of traits $S'\ra S$, with $K'=\Frac \Gamma(S,\O_S)$, the map $G(S')\ra  G(K')$ is surjective.
\end{itemize}
\end{proposition}

\begin{lemma} \label{Pic-N_etale}
Let $\X\ra S$ be a nodal curve over a trait. Let $\cll(e_K)\subset \Pic_{\X/S}$ be the schematic closure of the unit section $e_K\colon \Spec K\ra \Pic_{\X_K/K}$. Then the fppf-quotient sheaf $\mathcal N=\Pic_{\X/S}/\cll(e_K)$ is representable by a smooth separated $S$-group scheme. Moreover, the quotient morphism $\Pic_{\X/S}\ra \mathcal N$ is \'etale.
\end{lemma}
\begin{proof}

As $\cll(e_K)$ is flat over $S$, the fppf-quotient of sheaves $\mathcal N=\Pic_{\X/S}/\cll(e_K)$ is a group algebraic space, smooth over $S$ because $\Pic_{\X/S}$ is; as $\cll(e_K)$ is closed in $\Pic_{\X/S}$, $\mathcal N$ is separated over $S$. 
In particular, $\mathcal N$ is a separated group algebraic space locally of finite type over $S$, so it is a group scheme by \citep{Anantharaman1973}, Chapter IV, Theorem 4.B. Finally, to show that $\Pic_{\X/S}\ra \mathcal N$ is \'etale we prove that $\cll(e_K)$ is \'etale over $S$. As the property is \'etale local on $S$, we may assume that $\X\ra S$ has special fibre with split singularities. The multidegree map $E(\X)\ra \Z^{V}$ (Lemma \ref{comparison}, ii)) is injective, hence the intersection of $\cll(e_K)$ with the identity component $\Pic^0_{\X/S}\sub \Pic_{\X/S}$ is trivial and it follows that $\cll(e_K)$ is \'etale over $S$.
\end{proof}

Given a nodal curve $\X\ra S$ over a trait, we can associate to it the labelled graph $(\Gamma,l)$ of the base change $\X\times_SS'\ra S'$, where $S'$ is the spectrum of the strict henselization of $\Gamma(S,\O_S)$ with respect to some algebraic closure of the residue field $k$. The graph $(\Gamma,l)$ does not depend on the choice of an algebraic closure of $k$.
\begin{theorem} \label{mainneron}
Let $\X\ra S$ be a nodal curve over a trait. The $S$-group scheme $\mathcal N=\Pic_{\X/S}/\cll(e_K)$ is a N\'eron lft-model for $\Pic_{\X_K/K}$ over $S$ if and only if the labelled graph $(\Gamma,l)$ of $\X\ra S$ is circuit-coprime.  
\end{theorem}
\begin{proof}
Let $S^{sh}\ra S$ be a strict henselization of $S$ with respect to some algebraic closure of the residue field, and denote by $K^{sh}$ its fraction field. If $(\Gamma,l)$ is not circuit-coprime, the map 
$$\Pic(\X_{S^{sh}})\rightarrow \Pic(\X_{K^{sh}})$$
is not surjective, by Theorem \ref{mainthm}. Now, as the special fibre of $\X_{S^{sh}}/S^{sh}$ is generically smooth, $\X_{S^{sh}}\ra S^{sh}$ admits a section; hence, we can apply Lemma \ref{rigidif} and find that 
$$\Pic_{\X/S}(S^{sh})\ra \Pic_{\X_K/K}(K^{sh})$$
is not surjective. As the quotient $\Pic_{\X/S}\ra \mathcal N$ is an \'etale surjective morphism of $S^{sh}$-algebraic spaces (Lemma \ref{Pic-N_etale}), the map $\Pic_{\X/S}(S^{sh})\ra \mathcal N(S^{sh})$ is surjective.  We deduce that $\mathcal N(S^{sh})\ra \Pic_{\X_K/K}(K^{sh})$
is not surjective.  Then for some \'etale extension of discrete valuation rings $S'\ra S$, $\mathcal N(S')\ra \Pic_{\X_K/K}(K')$ is not surjective, hence $\mathcal N$ is not a N\'eron model of $\Pic_{\X_K/K}$.

Now assume that $(\Gamma,l)$ is circuit coprime. Assume first that $S$ is strictly henselian. 
By Prop. \ref{weakNM} it is enough to prove that for all essentially smooth local extensions $R\ra R'$ of discrete valuation rings, the map 
$$\mathcal N(R')\ra \Pic_{\X_K/K}(K')$$
is surjective. As $\X\ra S$ admits a section, we may apply Lemma \ref{rigidif} and just show that $\Pic(\X_{R'})\rightarrow \Pic(\X_{K'})$ is surjective. The map $R\ra R'$ has ramification index $1$, i.e. it sends a uniformizer to a uniformizer. Therefore the labelled graph $(\Gamma',l')$ associated to $\X_{R'}$ is again circuit-coprime, and in fact $(\Gamma',l')=(\Gamma,l)$. Now we conclude by Theorem \ref{mainthm}.

Now let $\X\ra S$ be any nodal curve with circuit-coprime labelled graph. Let $p\colon S'\ra S$ be a strict henselization of $S$. Consider the smooth separated $S$-group scheme $\mathcal N=\Pic_{\X/S}/\cll(e_K)$. As taking the schematic closure commutes with flat base change, $p^*\mathcal N$ is canonically isomorphic to $\Pic_{\X'/S'}/\cll(e_{K'})$, hence is a N\'eron lft-model for $\Pic_{\X_{K'}/K'}$ over $S'$. We show that $\mathcal N$ is a N\'eron lft-model of its generic fibre. Let $T\ra S$ be a smooth $S$-scheme, $f\colon T_K\ra \Pic_{\X_K/K}$ a $K$-morphism. The base change $p^*f\colon T_{K'}\ra \Pic_{\X_{K'}/K'}$ extends uniquely to an $S'$-morphism $g\colon p^*T\ra \mathcal N'$. Let $S'':=S'\times_SS'$, $p_1,p_2\colon S''\ra S'$ the two projections, and $q\colon S''\ra S$ the composition. The two maps $p_1^*g, p_2^*g\colon q^*T\ra q^*\mathcal N$ both coincide with $q^*f$ when restricted to $q^*T_K$. As $q^*T\ra S$ is flat, $q^*T_K$ is schematically dense in $q^*T$. Since moreover $q^*\mathcal N$ is separated, we have that $p_1^*g=p_2^*g$. Hence $g$ descends to a morphism $T\ra \mathcal N$ extending $f$. Again, the extension is unique because $\mathcal N\ra S$ is separated and $T_K$ is schematically dense in $T$.  


\end{proof}

\begin{corollary}
Let $\X\ra S$ be a nodal curve over a trait. Let $\pi\colon \widetilde{\X}\ra \X$ be the blowing-up of $\X$ at the finite union of closed points $\X^{nreg}\cap \X_k$. Then 
$\mathcal N=\Pic_{\widetilde{\X}/S}/\cll(e_K)$ is a N\'eron lft-model for $\Pic_{\X_K/K}$ over $S$.
\end{corollary}
\begin{proof}
It is enough to check that the labelled graph $(\til \Gamma, \til l)$ of $\til \X\ra S$ is circuit-coprime, by the previous Theorem. As labelled graphs are preserved under \'etale extensions of the base trait, we may assume that $\X\ra S$ has special fibre with split singularities. Then the same argument as in the proof of Corollary \ref{blowuppp} shows that $(\til \Gamma, \til l)$ is circuit-coprime.
\end{proof}

\begin{corollary}
Let $\X\ra S$ be a nodal curve over a trait. Let $\overline k$ be a separable closure of the residue field of $S$ and suppose that the graph of $\X_{\overline k}$ is a tree. Then 
$\mathcal N=\Pic_{\X/S}/\cll(e_K)$ is a N\'eron lft-model for $\Pic_{\X_K/K}$ over $S$.
\end{corollary}

We have shown how to construct N\'eron lft-models for the group scheme $\Pic_{\X_K/K}$, without ever imposing bounds on the degree of line bundles; the following lemma allows us to retrieve lft-N\'eron models for subgroup schemes of $\Pic_{\X_K/K}$, and applies in particular to subgroup schemes that are open and closed, such as the connected component of the identity $\Pic^{[0]}_{\X_K/K}$. 

\begin{lemma}\label{Pic0}
Let $\X/S$ be a nodal curve over a trait, and $H\subset \Pic_{\X_K/K}$ a $K$-smooth closed subgroup scheme of $\Pic_{\X_K/K}$. Let $\mathcal N\ra S$ be the N\'eron model of $\Pic_{\X_K/K}$. Then $H$ admits a N\'eron lft-model $\mathcal H$ over $S$, which is obtained as a group smoothening of the schematic closure of $H$ inside $\mathcal N$.
\end{lemma}
\begin{proof}
This is a special case of \citep{BLR}, 10.1/4.

\end{proof}

We remark that, if the generic fibre $\X_K/K$ is not smooth, $\Pic^{[0]}_{\X_K/K}$ is an extension of an abelian variety by a torus; if the torus contains a copy of $\mathbb G_{m,K}$, the N\'eron lft-model of $\Pic^{[0]}_{\X_K/K}$ is not quasi-compact.

\nocite{*}
\bibliographystyle{plain}

\end{document}